\documentclass[reqno]{amsart}

\usepackage{begnac}

\renewcommand{\det}{\operatorname{det}}
\newcommand{\brbinom}[2]{{#1 \brack #2}}

\mynewop{loc}

\mynewop{lindim}

\begin{document}

\title[Everything I always wanted to know]{Everything I always wanted to know about resultants and Chow forms* \\
  \small (*But was too lazy to ask)}

\author{Itaï \textsc{Ben Yaacov}}

\address{Itaï \textsc{Ben Yaacov} \\
  Univ Lyon \\
  Université Claude Bernard -- Lyon 1 \\
  Institut Camille Jordan, CNRS UMR 5208 \\
  43 boulevard du 11 novembre 1918 \\
  69622 Villeurbanne Cedex \\
  France}

\urladdr{\url{http://math.univ-lyon1.fr/~begnac/}}

\keywords{Resultant; Chow form; Intersection degree}

\begin{abstract}
  This note develops some fundamental properties of resultants and related notions.
  It represents my own personal exploration of this domain, which I found more instructive than seeking answers in the standard literature.
  Consequently, notation and terminology may be quite idiosyncratic, and the approach is very algebraic.

  Read at your own risk.
\end{abstract}

\maketitle

\tableofcontents

\section*{Introduction}

For a standard reference regarding resultants and intersections, see Gelfand, Kapranov and Zelevinsky~\cite{Gelfand-Kapranov-Zelevinsky:Discriminants}.

When working in ambient projective dimension $n$, we let $X = (X_0,\ldots,X_n)$ denote indeterminates representing homogeneous coordinates.
For $\alpha \in \bN^{n+1}$, we define the corresponding monomial as $X^\alpha = \prod X_s^{\alpha_s}$.
We also consider the \emph{dual indeterminates} $X^* = (X^*_0,\ldots,X^*_n)$, and let $X^* \cdot X = \sum X^*_s X_s$.
We may think of $X^* \cdot X$ as an indeterminate linear form in $X$, with indeterminate coefficients $X^*$.
We may think of $X^*$ as a row matrix, and of $X$ as a column, so $X^* \cdot X$ is just the matrix product $X^* X$.

Having defined $X^\alpha$ and $X^* \cdot X$, we no longer need to reference individual coordinates of either $X$ or $X^*$.
The notation $X_i$ or $X^*_i$ will be reserved (unless explicitly stated otherwise) for entire copies of $X$ or $X^*$, respectively.

Throughout, we are going to work in one of two settings.
\begin{itemize}
\item In the \emph{algebraic setting}, we work in the category of unital commutative rings.
  Working \emph{over} a ring $A$ means working in the category of $A$-algebras, namely of rings $B$ equipped with a morphism from $A$.
\item Alternatively, in the (Weil-style) \emph{algebraic geometry setting} we work over a field $k$.
  The category of $k$-algebras can be replaced with a single field extension $K$ that is algebraically closed and of infinite transcendence degree over $k$.
\end{itemize}

If $A$ is a ring (always commutative, with unit), then $A[X] = \bigoplus_d A[X]_d$ is a graded ring, where $A[X]_d$ denotes the set of homogeneous polynomials of degree $d$.

\section{Special kinds of polynomials}

\subsection{Symmetric polynomials}

\begin{ntn}
  \label{ntn:Monomials}
  We fix a projective dimension $n$.
  We denote by $\bN_d$ the set of $\alpha \in \bN^{n+1}$ such that $\sum \alpha_i = d$.
  In other words, the monomials of degree $d$ in $X$ are $\{X^\alpha : \alpha \in \bN_d\}$.
\end{ntn}

\begin{ntn}
  Let $D \in \bN$, let $X_i$ be copies of $X$, and let $X_{<D} = (X_i : i < D)$.
  Let $d \in \bN$.
  \begin{itemize}
  \item We define $\fM_{D,d}$ as the set of all monomials in $X_{<D}$ that are homogeneous of degree $d$ in each $X_i$.
  \item For $\fm \in \fM_{D,d}$, define $P_\fm = \sum \fS_D \cdot \fm$ (notice that we take the sum of the set, so each monomial appears only once, even if it is obtained by more than one permutation).
  \end{itemize}
\end{ntn}

\begin{dfn}
  \label{dfn:Symmetric}
  A \emph{homogeneous symmetric polynomial} of degree $d$ in $D$ $n$-dimensional indeterminates, over a ring $A$, is a polynomial $P \in A[X_{<D}]$ that is homogeneous of degree $d$ in each group $X_i$, and invariant under the action of the symmetric group $\fS_D$.
  Since the dimension $n$ is fixed, we usually omit it.

  Equivalently, the collection of homogeneous symmetric polynomials of degree $d$ in $D$ indeterminates over $A$ is the free $A$-module generated by the set $\{ P_\fm : \fm \in \fM_{D,d}\}$.
\end{dfn}

Let us define polynomials $\sigma_\alpha$ for $\alpha \in \bN_D$ by
\begin{gather}
  \label{eq:GD}
  G_D(X^*) = \prod_{i<D} \, (X^* \cdot X_i) = \sum_{\alpha \in \bN_D} \sigma_\alpha {X^*}^\alpha.
\end{gather}
Each $\sigma_\alpha$ is symmetric and homogeneous of degree one (in each $X_i$).

\begin{dfn}
  \label{dfn:ElementarySymmetric}
  The polynomials $\{\sigma_\alpha : \alpha \in \bN_D\}$ are the \emph{elementary} homogeneous symmetric polynomials in $D$ indeterminates.
\end{dfn}

For any monomial $\fm$ in the indeterminates $X_{<D}$ we may substitute $X$ for each $X_i$, and obtain $\delta(\fm) \in \bN^{n+1}$ as the multi-exponent in the identity $\fm(X,X, \ldots, X) = X^{\delta(\fm)}$.
If $\fm \in \fM_{D,d}$, then $\delta(\fm) \in \bN_{Dd}$.
If $\fm \in \fM_{D,1}$, then $P_\fm = \sigma_{\delta(\fm)}$.
The map $\delta\colon \fM_{D,1} \rightarrow \bN_D$ is onto and $\{\sigma_\alpha : \alpha \in \bN_D\}$ is an enumeration, without repetitions, of $\{P_\fm : \fm \in \fM_{D,1}\}$.

Observe that if $\fm,\fm' \in \fM_{D,1}$ are distinct, and $\delta(\fm) = \delta(\fm')$, then there exists a third monomial $\fm'' \in \fM_{D,1}$ such that $\fm'' \mid \lcm(\fm,\fm')$ and $\delta(\fm'') > \delta(\fm)$ in lexicographical order.
Consequently, for every $\fm \in \fM_{D,d+1}$ there exists a unique $\fm_0 \in \fM_{D,1}$ such that $\fm_0 \mid \fm$ and $\delta(\fm_0)$ is maximal.

\begin{fct}
  \label{fct:Symmetric}
  The collection of all homogeneous symmetric polynomials of degree $d$ in $D$ variables over $A$ is exactly $A[\sigma_D]_d$, namely, the collection of polynomials of degree $d$ in the elementary symmetric polynomials $\sigma_D$.
\end{fct}
\begin{proof}
  One inclusion is clear.
  For the other, it will suffice to show that $P_\fm \in \bZ[\sigma_D]_d$ for every $\fm \in \fM_{D,d}$, by induction on $d$.
  The case where $d = 0$ is clear.

  For the induction step, let $\fm \in \fM_{D,d+1}$.
  Choose a decomposition $\fm = \fm_0 \fm_1$ where $\fm_0 \in \fM_{D,1}$, $\fm_1 \in \fM_{D,d}$, and moreover, $\delta(\fm_0)$ is greatest possible in the lexicographical order on $\bN_D$.
  By the induction hypothesis, $P_{\fm_1} \in \bZ[\sigma_D]_d$.
  Since $P_{\fm_0} = \sigma_{\delta(\fm_0)}$, it will suffice to show that $P_{\fm_0} P_{\fm_1} - P_\fm \in \bZ[\sigma_D]_{d+1}$.

  By maximality of $\delta(\fm_0)$, the only monomial occurring in $P_{\fm_0}$ and dividing $\fm$ is $\fm_0$, so $\fm$ occurs in $P_{\fm_0} P_{\fm_1}$ exactly once, as $\fm_0 \fm_1$.
  Consequently, each monomial of $P_\fm$ occurs in $P_{\fm_0} P_{\fm_1}$ exactly once.
  Consider a monomial $\fn$ that occurs in the difference $P_{\fm_0} P_{\fm_1} - P_\fm$ -- that is to say that it occurs in $P_{\fm_0} P_{\fm_1}$ but not in $P_\fm$.
  Possibly replacing with another monomial in the same orbit $\fS_D \cdot \fn$, we have $\fn = \fn_0 \fm_1$, where $\fn_0 = \tau(\fm_0) \neq \fm_0$.
  Therefore $\delta(\fn_0) < \delta(\fm_0)$ and $\delta(\fn) < \delta(\fm)$.
  Adding an induction on $\delta(\fm)$, we may assume that $P_\fn \in \bZ[\sigma_D]_{d+1}$ for every such $\fn$, and the proof is complete.
\end{proof}

\subsection{Splittable polynomials}

\begin{dfn}
  \label{dfn:Splittable}
  Let $A$ be a ring (always commutative, with unit), $D > 0$, and
  \begin{gather*}
    g(X^*) = \sum_{\alpha \in \bN_D} a_\beta {X^*}^\alpha \in A[X^*]_D.
  \end{gather*}
  We say that $g$ is \emph{splittable} if its tuple of coefficients $a = (a_\alpha : \alpha \in \bN_D)$ satisfies every polynomial relation over $\bZ$ satisfied by $\sigma_D$.

  For $D = 0$, we consider every $g \in A[X^*]_0 = A$ to be splittable.
\end{dfn}

Equivalently, for $D \in \bN$, $g \in A[X^*]_D$ is splittable if and only if there exists a ring morphism $\varphi\colon \bZ[Y,\sigma_D] \rightarrow A$ that sends $Y G_D$ to $g$, where $G_D$ is as per \autoref{eq:GD}.
The splittability condition is homogeneous, so the zero polynomial in $A[X^*]_D$ is splittable.

\begin{conv}
  \label{conv:PointsAsPolynomials}
  Let $A$ be a ring (often, but not necessarily always, a field).
  There exists an obvious bijection between points $x \in A^{n+1}$ and linear homogeneous polynomials $X^* \cdot x \in A[X^*]_1$.
  We are going to identify them, so, in particular, by a product of points $\prod_{i<D} x_i$ we mean the polynomial $\prod_{i<D} \, (X^* \cdot x_i) \in A[X^*]_D$.
  We may do the same with any fixed distinguished copy of $X^*$ (e.g., $X^*_\ell$ in \autoref{sec:ChowForm}).
\end{conv}

With this convention, if $K$ is algebraically closed and $g \in K[X^*]_D$ splittable, then splits $g$ as $y \prod_{i<D} x_i$, with $y \in K$.
If $D > 0$, then we may require that $y = 1$ (but this does not make the choice of $x_i$ unique), and if $D = 0$, then $g = y$.

Thus, a non-zero splittable $g$ determines the multi-set (i.e., set with multiplicities, or a zero-dimensional chain) $[g] = \bigl\{ [x_i] : i < D \bigr\} \subseteq \bP^n(K)$.
Conversely, every non-empty finite multi-set in $\bP^n(K)$ is of the form $[g]$ for a splittable non-zero $g$ over $K$, unique up to a factor in $K^\times$.

\begin{dfn}
  \label{dfn:Wedge}
  Let $d,D \in \bN$.
  Let $T^* = (T^*_\alpha : \alpha \in \bN_d)$ be indeterminates representing the coefficients of a polynomial of degree $d$
  \begin{gather*}
    F(X) = \sum_{\alpha \in \bN_d} T^*_\alpha X^\alpha \in \bZ[T^*][X]_d.
  \end{gather*}
  Let $G_D(X^*) \in \bZ[\sigma_D][X^*]_D \subseteq \bZ[X_{<D}][X^*]_D$ be as per \autoref{eq:GD}.
  We define
  \begin{gather*}
    F \wedge G_D = \prod_{i<D} F(X_i).
  \end{gather*}
  Each coefficient in $F \wedge G_D$ of a monomial in $T^*$ is a symmetric homogeneous polynomial of degree $d$ in $X_{<D}$, so $F \wedge G_D \in \bZ[T^*,\sigma_D]$.

  Consider now a ring $A$, a polynomial $f \in A[X]_d$ and a splittable polynomial $g(X^*) \in A[X^*]_D$.
  Then there exists a morphism $\varphi\colon \bZ[T^*,Y,\sigma_D] \rightarrow A$ that sends $F \mapsto f$ and $YG_D \mapsto g$, and we define $f \wedge g = \varphi(Y)^d \varphi(F \wedge G_D) \in A$.
\end{dfn}

It is easy to check that the wedge operation is well defined, even though the homomorphism that sends $YG_D \mapsto g$ is not, in general, unique.
If $D > 0$, then we can make $\varphi$ unique by requiring that $\varphi(Y) = 1$, in which case $f \wedge g = \varphi(F \wedge G_D)$.
If $D = 0$, then $\varphi$ is unique, $\varphi(Y) = g \in A$, and $f \wedge g = g^d$.
Similarly, if $d = 0$, then $f \wedge g = f^D$; and if $d = D = 0$, then $f \wedge g = 1$.
As usual for polynomial expressions, we follow the convention that $0^0 = 1$.

The wedge operation is, by construction, functorial: if $\psi\colon A \rightarrow B$ is a morphism, and $f,g$ are as in \autoref{dfn:Wedge}, then $\psi(g) \in B[X^*]_D$ is splittable, and $\psi(f) \wedge \psi(g) = \psi(f \wedge g) \in B$.
If $g$ splits as $y \prod x_i$ over $A$ (or, by functoriality, over an extension of $A$), then $f \wedge g = y^d \prod f(x_i)$.

\begin{lem}
  \label{lem:Wedge}
  Let $f \in A[X]_d$ and $g \in A[X^*]_D$, with $g$ splittable.
  Then
  \begin{enumerate}
  \item
    \label{item:SplittableFactor}
    If $g = g_1 g_2$ and both $g_i$ are splittable, then so is $g$.
    If $A$ is an integral domain and $g \neq 0$, then the converse holds as well.
  \item
    \label{item:WedgeProduct}
    If $f = f_1 f_2$, or $g = g_1g_2$, and both $g_i$ are splittable, then
    \begin{gather*}
      (f_1 f_2) \wedge g = (f_1 \wedge g)(f_2 \wedge g),
      \qquad
      f \wedge (g_1 g_2) = (f \wedge g_1)(f \wedge g_2).
    \end{gather*}
  \item
    \label{item:WedgeAddition}
    Assume that $A$ is an integral domain and $g \neq 0$.
    Let $h$ be an indeterminate polynomial of degree $d$ (i.e., with indeterminate coefficients, that we may adjoin to $A$).
    Then $f$ vanishes at every linear factor of $g$ (over the algebraic closure of $\Frac(A)$) if and only if
    \begin{gather*}
      (f + h) \wedge g = h \wedge g.
    \end{gather*}
  \end{enumerate}
\end{lem}
\begin{proof}
  For \autoref{item:SplittableFactor}, one direction is clear, and for the converse, split $g$ over $\Frac(A)^a$.
  Items \autoref{item:WedgeProduct} and \autoref{item:WedgeAddition} are immediate from the definitions.
\end{proof}

\subsection{Alternating polynomials}

Let $A$ be a ring, and $\ell \in \bN$.
Let $(X^*_i : i \leq \ell)$ consist of $\ell+1$ copies of $X^*$.
At this stage we \emph{could} work with copies of $X$ rather than of $X^*$, and with $\ell$ many copies rather than $\ell+1$, but this setting fits better with the manner in which alternating polynomials are used later on.

\begin{dfn}
  \label{dfn:Alternating}
  A polynomial $g \in A[X^*,Y^*]$ is \emph{alternating} if, for a new indeterminate $Z$,
  \begin{gather}
    \label{eq:Alternating}
    g(X^*, Y^*) = g(X^* + Z Y^*,Y^*) = g(X^*, Y^* + Z X^*).
  \end{gather}
  A polynomial $g \in A[X^*_i : i \leq \ell]$ is alternating if it is as a polynomial in $X^*_i$ and $X^*_j$ for every $i \neq j$.
  We say that $g$ is \emph{homogeneous} of degree $d$ as an alternating polynomial, if it is so as a polynomial in $X^*_\ell$ over $A[X^*_i : i < \ell]$.

  If we want to make $\ell$ explicit, we say that $g$ is alternating \emph{in dimension $\ell$}.
  We may sometimes allow $\ell = -1$, in which case we consider every $g \in A[X^*_i : i \leq -1] = A$ to be alternating, of any desired degree.
\end{dfn}

If $g$ is alternating, then successive applications of the definition yield
\begin{gather}
  \label{eq:AlternatingExchange}
  g(X^*,Y^*)
  = g(X^* + Y^*,Y^*)
  = g(X^* + Y^*,-X^*)
  = g(Y^*,-X^*).
\end{gather}


\begin{lem}
  \label{lem:AlternatingFactor}
  A product of alternating polynomials in $A[X^*_i : i \leq \ell]$ is again alternating.
  If $A$ is an integral domain, then every factor of an alternating polynomial is also alternating.
\end{lem}
\begin{proof}
  The first assertion is immediate.
  For the second, it is enough to consider the case where $f,g \in A[X^*,Y^*]$ and $h = fg$ is alternating.
  Then $Z$ cannot appear in the development of $f(X^*+Z Y^*,Y^*)$, since the latter is a factor of $h(X^*,Y^*)$.
  Substituting $Z = 0$, we see that $f$ is alternating.
\end{proof}

If $g \in A[X^*,Y^*]$ is homogeneous in each of $X^*$ and $Y^*$, then \autoref{eq:Alternating} is equivalent to the same with $Z = 1$.
Assume that $g \in A[X^*,Y^*]$ is alternating, though not necessarily homogeneous.
Let $g_d$ be its homogeneous component of degree $d$ in $X^*$, and let $g_{d,e}$ be the homogeneous component of $g_d$ of degree $e$ in $Y^*$, so $g = \sum_d g_d = \sum_{d,e} g_{d,e}$.
Stratifying \autoref{eq:Alternating} according to homogeneous degree in $X^*,Z$, we must have $g_d(X^* + Z Y^*,Y^*) = g_d(X^*,Y^*)$ for each $d$.
Stratifying the latter according to homogeneous degree in $X^*,Y^*$, we then obtain $g_{d,e}(X^* + Z Y^*,Y^*) = g_{d,e}(X^*,Y^*)$.
Applying the same logic with $X^*$ and $Y^*$ exchanged, we see that $g_{d,e}$ must satisfy \autoref{eq:Alternating}.
Therefore (by \autoref{eq:AlternatingExchange}) $g_{d,e} = 0$ if $d \neq e$, and $g = \sum g_{d,d}$ is a sum of homogeneous alternating polynomials.
The same applies for an alternating polynomial in arbitrarily many groups of indeterminates: it is the sum, over $d$, of alternating polynomials that are homogeneous of degree $d$ in each group separately.

\begin{dfn}
  \label{dfn:AlternatingSharp}
  Let $m \geq 1$, and let $g \in A[X^*_{<m}]$ be alternating and homogeneous of degree $d$.
  Let $X^*_m$ be an additional copy of $X^*$, and let $Y = (Y_i : i \leq m)$ be additional indeterminates.
  We define
  \begin{gather*}
    g^\sharp(X^*_{\leq m},Y)
    = g\left( \frac{X^*_i}{Y_i} - \frac{X^*_m}{Y_m} : i < m \right) \prod_{i\leq m} Y_i^d.
  \end{gather*}
\end{dfn}

\begin{lem}
  \label{lem:AlternatingSharp}
  Let $g \in A[X^*_{<m}]$ be alternating and homogeneous of degree $d$.
  \begin{enumerate}
  \item The expression $g^\sharp(X^*_{\leq m},Y)$, as defined in \autoref{dfn:AlternatingSharp}, is a polynomial in $A[X^*_{\leq m},Y]$ (and not merely a rational function).
  \item
    \label{item:AlternatingSharpAlternate}
    For every $j \leq m$:
    \begin{gather*}
      g^\sharp(X^*_{\leq m},Y)
      = (-1)^{(m-j)d} g\left( X^*_i - \frac{Y_i}{Y_j} X^*_j : i \leq m, \ i \neq j \right) Y_j^d.
    \end{gather*}
  \item As a polynomial in $Y$ over $A[X^*_{\leq m}]$, $g^\sharp$ is homogeneous of degree $d$, and the coefficient of $Y_i^d$ is $\pm g(X^*_{\neq i})$ (or, to be precise, $(-1)^{(m-i)d} g(X^*_0,\ldots,X^*_{i-1},X^*_{i+1},\ldots,X^*_{})$).
  \item As a polynomial in $X^*_{\leq m}$ over $A[Y]$, $g^\sharp$ is homogeneous of degree $dm$.
  \item Letting $W^*_i = (X^*_i,Y_i)$ for $i \leq m$, the polynomial $g^\sharp \in A[W^*_{\leq m}]$ is alternating and homogeneous of degree $d$ in the family $(W_i : i \leq m)$.
  \item
    \label{item:AlternatingSharpIrreducible}
    If $A$ is a unique factorisation domain and $g$ is irreducible in $A[X^*_{<m}]$, then $g^\sharp$ is irreducible in $A[X^*_{\leq m},Y]$.
  \end{enumerate}
\end{lem}
\begin{proof}
  Let $j < m$.
  Subtracting $X^*_j/Y_j - X^*_m/Y_m$ from all the other arguments in the definition of $g^\sharp$ we do not change the result.
  Replacing $X^*_j/Y_j - X^*_m/Y_m$ with $X^*_m/Y_m - X^*_j/Y_j$ multiplies by $(-1)^d$, and pushing it to the end multiplies by another $(-1)^{(m-j-1)d}$.
  Finally, factoring each $Y_i^d$, for $i \neq j$, into the corresponding argument of $g$, we obtain \autoref{item:AlternatingSharpAlternate}.
  In particular, $g^\sharp \in A[X^*_{\leq m},Y,Y_j^{-1}]$, and considering $j = 0,1$ we conclude that $g^\sharp \in A[X^*_{\leq m},Y]$.
  The homogeneity properties and calculation of the coefficients of $Y_i^d$ follow.

  Assume now that $A$ is a unique factorisation domain and $g \in A[X^*_{<m}]$ is irreducible.
  Let $f,h \in A[X^*_{\leq m},Y]$ be such that $fh = g^\sharp = g(X^*_{<m}) Y_m^d + \cdots$.
  Since $g(X^*_{<m})$ is irreducible, we may assume that $f = g(X^*_{<m}) Y_m^k + \cdots$ and $h = Y_m^{d-k} + \cdots$.
  Then $f$ must be homogeneous of degree $dm$ in $X^*_{\leq m}$, and $h \in A[Y]_{d-k}$.

  Let $K$ denote the algebraic closure of $\Frac(A)$.
  If $k < d$, then $h$ admits a zero $[z] \in \bP^m(K)$, say with $z_j \neq 0$.
  But then the family of $X^*_i - z_i X^*_j/z_j$, for $i \neq j$, is generic over $K$, so $g^\sharp(X^*_{\leq m},z) \neq 0$, a contradiction.
  Therefore $k = d$, so $h \in A$, and it divides all the coefficients of $g^\sharp$.
  Since $g(X^*_{<m})$ is irreducible, $h$ is a unit of $A$.
  This completes the proof.
\end{proof}

Let us consider the indeterminates $X^*_i$ as the rows of an indeterminate $(\ell+1) \times (n+1)$ matrix denoted $\bX^*$.

\begin{lem}
  \label{lem:AlternatingDeterminant}
  A polynomial $g \in A[X^*_i : i \leq \ell] = A[\bX^*]$ is alternating and homogeneous (as such) of degree $d$ if and only if, for any $(\ell+1) \times (\ell+1)$ matrix $M$, possibly over a larger ring, we have
  \begin{gather}
    \label{eq:AlternatingDeterminant}
    g(M \bX^*) = \det M^d g(\bX^*).
  \end{gather}
\end{lem}
\begin{proof}
  Assume first that \autoref{eq:AlternatingDeterminant} holds for every $M$.
  Considering it, in particular, for the elementary matrices that multiply $X^*_i$ by $Z$, or add $Z X^*_j$ to $X^*_i$, we see that $g$ is homogeneous of degree $d$ in each $X^*_i$, and that it is alternating.

  For the converse, we proceed by induction on $\ell$, where the case $\ell = 0$ is immediate.
  For $\ell > 0$, consider a matrix $M = (a_{i,j})$.
  Let $M_1$ (respectively, $M_2$) be the result of replacing each row $L_i$, for $i < \ell$, with $a_{\ell,\ell} L_i$ (respectively, $a_{\ell,\ell} L_i - a_{i,\ell} L_\ell$).
  Then
  \begin{gather*}
    M_2
    =
    \begin{pmatrix}
      N & 0 \\
      ? & a_{\ell,\ell}
    \end{pmatrix}
    =
    \begin{pmatrix}
      N & 0 \\
      0 & 1
    \end{pmatrix}
    M_3,
    \qquad
    M_3 =
    \begin{pmatrix}
      I_\ell & 0 \\
      ? & a_{\ell,\ell}
    \end{pmatrix},
  \end{gather*}
  and
  \begin{gather*}
    g(M_2 \bX^*) = g(M_1 \bX^*) = a_{\ell,\ell}^{\ell d} g(M \bX^*).
  \end{gather*}
  On the other hand, $g$ is a homogeneous alternating polynomial in just $(X^*_i : i < \ell)$, so by the induction hypothesis:
  \begin{gather*}
    g(M_2 \bX^*) = \det N^d g(M_3 \bX^*) = (a_{\ell,\ell} \det N)^d g(\bX^*).
  \end{gather*}
  Finally, $a_{\ell,\ell} \det N = \det M_2 = \det M_1 = a_{\ell,\ell}^\ell \det M$.
  We may take $a_{\ell,\ell}$ to be a new indeterminate, by which we may cancel, and we obtain \autoref{eq:AlternatingDeterminant}.
\end{proof}

For $\ell > n$, it follows that the only alternating polynomials are the constants.
If $\ell = n$, a homogeneous alternating polynomial must be of the form $a (\det \bX^*)^d$, so an alternating polynomial is just a polynomial in $\det \bX^*$.
For $\ell = 0$, every polynomial in $A[X^*_0]$ is alternating.

Consider the case where $\ell = n-1$.
Let $g \in A[X^*_i : i \leq \ell] = A[X^*_i : i < n]$ be alternating homogeneous of degree $d$ and $g^\sharp \in A[Y^*_i,Z_i : i \leq n]$.
Let $z$ be the (column) vector $(0,\ldots,0,1)$, and let $M$ be the matrix adjugate to the square matrix $\bY^*$.
Then $Mz$ is the vector of maximal rank minors of $(Y^*_i : i < n)$, which we write as $Y^*_0 \wedge \cdots \wedge Y^*_{n-1}$.
Therefore
\begin{gather*}
  g^\sharp(M\bY^*,Mz)
  = g^\sharp\bigl( \det(\bY^*) I, Y^*_0 \wedge \cdots \wedge Y^*_{n-1} \bigr)
  = \det(\bY^*)^{dn} g^\sharp( I, Y^*_0 \wedge \cdots \wedge Y^*_{n-1} ).
\end{gather*}
On the other hand, by alternation of $g^\sharp$,
\begin{gather*}
  g^\sharp(M\bY^*,Mz)
  = \det(M)^d g^\sharp(\bY^*,z)
  = \det(\bY^*)^{dn} g(Y^*_i : i < n).
\end{gather*}
The determinant of an indeterminate matrix is never a zero divisor (more generally, if $f \in A[X]$ is such that $a f \neq 0$ for all $a \in A \{0\}$, then $gf \neq 0$ for all $g \in A[X] \setminus \{0\}$).
Therefore, when $m = n-1$, an alternating homogeneous polynomial $g$ is necessarily a polynomial function of the maximal rank minors, given by
\begin{gather*}
  g(\bX^*) = g^\sharp(I, X^*_0 \wedge \cdots \wedge X^*_m).
\end{gather*}

In the general case, the minors of $\bX^*$ of maximal rank are linear alternating polynomials (by \emph{linear}, we mean in each $X^*_i$ separately).
Call them the \emph{elementary} alternating polynomials in $X^*_0,\ldots,X^*_m$.
Let $\brbinom{n+1}{m+1}$ denote the collection of $(m+1)$-element subsets of $n+1 = \{0,\ldots,n\}$, so the elementary alternating polynomials may be naturally indexed as $\fa = \bigl\{ \fa_s(\bX^*) : s \in \brbinom{n+1}{m+1} \bigr\} \subseteq \bZ[\bX^*]$.
Then one can show that for any $s \in \brbinom{n+1}{m+1}$, every alternating polynomial $g(\bX^*)$ belongs to the ring $A[\fa,\fa_s^{-1}]$, and is moreover integral in $A[\bX^*]$ over $A[\fa]$.
If $A$ is an integrally closed domain, then the set of alternating polynomials is exactly the integral closure of $A[\fa]$.

\section{Resultants}
\label{sec:Resultant}

In this section we work in a Weil-style algebraic geometric setting.
We fix a field $k$, together with an algebraically closed extension $L$ of infinite transcendence degree.
All points, linear forms, polynomials, and so on, are $L$-rational.

\begin{ntn}
  \label{ntn:DimDeg}
  If $X$ is a collection of projective or affine $L$-rational points, let us write $\dim_k X$ for the transcendence degree of $k(X)$ over $k$.
  If $\dim_k X = 0$, then $\deg_k X$ is the degree of the algebraic extension $k(X)/k$.

  If we want to consider the linear dimension of a vector space, we shall write $\lindim$ explicitly.
\end{ntn}

Each $L$-linear space of homogeneous polynomials $L[X]_d$ can be viewed as an affine variety, isomorphic to $L^{\binom{n+d}{n}}$.

\begin{ntn}
  \label{ntn:VStar}
  Let $\xi = [x] \in \bP^n(L)$.
  We denote by $V^*_d(\xi)$ the $L$-linear space of homogeneous polynomials of degree $d$ that vanish at $\xi$, namely
  \begin{gather*}
    V^*_d(\xi) = \{ f \in L[X]_d : f(x) = 0\}.
  \end{gather*}
  When $d = 1$, we may drop it.
\end{ntn}

\begin{conv}
  \label{conv:ResultantSetup}
  Throughout, we let $I$ denote a finite set of indices.
  A family indexed by $I$ will be denoted $a_I = (a_i : i \in I)$.
  We let $n_I = (n_i : i \in I)$ be a family of projective dimensions, and let $\bP^{n_I}(L) = \prod_{i \in I} \bP^{n_i}(L)$.
  We let $X_i$ be indeterminates homogeneous coordinates in $\bP^{n_i}$, and let $X^*_i$ be the corresponding dual indeterminates.

  For a family $\xi_I = (\xi_i : i \in I) \in \bP^{n_I}(L)$, we define
  \begin{gather*}
    V^*(\xi_I) = \prod_{i \in I} V^*(\xi_i) \subseteq \prod_{i \in I} L[X_i]_1.
  \end{gather*}
  This is an irreducible variety of dimension $\sum_{i \in I} n_i$, defined over $k(\xi_I)$.
\end{conv}

\begin{rmk}
  When we do not care to fix the projective dimension, we let
  \begin{gather*}
    \bP(L) = \bigcup_{n \in \bN} \bP^n(L),
    \qquad
    \text{so}
    \qquad
    \bP(L)^I = \bigcup_{n_I \in \bN^I} \bP^{n_I}(L).
  \end{gather*}
\end{rmk}

\begin{dfn}
  \label{dfn:Constrained}
  Let $\xi_I \in \bP(L)^I$ and let $\lambda_I \in V^*(\xi_I)$ be generic over $k(\xi_I)$.
  \begin{enumerate}
  \item We say that $\xi_I$ is \emph{unconstrained} over $k$ if $\lambda_I$ is algebraically free over $k$.
    Otherwise, $\xi_k$ is \emph{constrained} over $k$.
  \item We say that $\xi_I$ is \emph{singly constrained} if $\lambda_I$ satisfies a single irreducible algebraic relation over $k$.
  \item We say that $\xi_I$ is \emph{multiply constrained} if $\lambda_I$ satisfies several algebraic relations over $k$.
  \end{enumerate}
\end{dfn}

\begin{dfn}
  \label{dfn:Resultant}
  Let $\xi_I \in \bP(L)^I$ be singly constrained, and let $\lambda_I \in V^*(\xi_I)$ be generic over $k(\xi_I)$.
  Then we define $\fR^k_{\xi_I} \in k[X^*_I]$, the \emph{resultant} associated to $\xi_I$ over $k$, to be the unique irreducible algebraic relation over $k$ satisfied by $\lambda_I$.
\end{dfn}

\begin{conv}
  \label{conv:ResultantEquivalence}
  When $\xi_I$ is singly constrained over $k$, \autoref{dfn:Resultant} only determines $\fR^k_{\xi_I}$ up to a multiplicative factor in $k^\times$.
  We are going to say that two expressions involving resultants are \emph{equivalent}, denoted by $\equiv$, if it is possible to choose representatives (usually for several resultants, sometimes relative to different fields) so that equality holds.
\end{conv}

\begin{rmk}
  \label{rmk:ResultantReduction}
  Let $J \subseteq I$ and $I_0 = I \setminus J$.
  Assume that $\xi_I \in \bP(L)^I$, and that $\xi_J$ is unconstrained.
  Let $\lambda_I \in V^*(\xi_I)$ be generic, and let $k_1 = k(\lambda_J)$.
  \begin{enumerate}
  \item
    The family $\xi_{I_0}$ is (singly, multiply) constrained over $k_1$ if and only if $\xi_I$ is over $k$.
  \item
    \label{item:ResultantReductionIsomorphism}
    There exists a natural isomorphism $k(X^*_J) \simeq k_1$, identifying $X^*_j$ with the coefficients of $\lambda_j$.
  \end{enumerate}
  Assume moreover that $\xi_I$ is singly constrained over $k$, so $\xi_{I_0}$ is singly constrained over $k_1$.
  Then, under the identification of \autoref{item:ResultantReductionIsomorphism}, we have $k[X^*_I] \subseteq k_1[X^*_{I_0}]$, and $\fR^k_{\xi_I} \equiv \fR^{k_1}_{\xi_{I_0}}$ (that is to say that any instance of $\fR^k_{\xi_I}$ is also an instance of $\fR^{k_1}_{\xi_{I_0}}$, even though the first is well defined up to a factor in $k^\times$, and the second up to $k_1^\times$).
\end{rmk}

\begin{thm}
  \label{thm:ResultantSplitting}
  Let $\xi_I \in \bP(L)^I$, and let $\lambda_I \in V^*(\xi_I)$ be generic over $k(\xi_I)$.
  Let $i \in I$ and $J = I \setminus \{i\}$, and assume that $\xi_I$ is constrained over $k$ but $\xi_J$ is not.

  Then the family $\xi_I$ is singly constrained over $k$, and the projective point $\xi_i$ is algebraic over $k(\lambda_J)$ and rational over $k(\lambda_I)$.

  Moreover, let $\bigl\{ [x_t] : t < D_s \bigr\}$ be the distinct conjugates of $\xi$ over $k(\lambda_J)$, where $D_s = \bigl[ k(\xi_i,\lambda_J) : k(\lambda_J) \bigr]_s$ is the separable degree.
  Let $D_i = \bigl[ k(\xi_i,\lambda_J) : k(\lambda_J) \bigr]_i$ be the inseparable degree.
  Then
  \begin{gather*}
    \deg_{X^*_i} \fR^k_{\xi_I} = \bigl[ k(\xi_i,\lambda_J) : k(\lambda_J) \bigr] = D_s D_i
    \qquad \text{and} \qquad
    \fR^k_{\xi_I} \equiv \prod_{t<D_s} \, (X^*_i \cdot x_t)^{D_i}.
  \end{gather*}
\end{thm}
\begin{proof}
  By \autoref{rmk:ResultantReduction}, we may assume that $I = \{i\}$, so $\xi_I = \xi_i = \xi$ and $\lambda_I = \lambda_i = \lambda$.
  We may also assume that $\xi = [x]$, where $x = (y,1)$.
  Accordingly, say that $X^* = (Y^*,Z^*)$, where $Z^*$ is a singleton.

  Let $k_1 = k(Y^*)$ and $z = - Y^* \cdot y \in k_1(\xi) = k_1(y)$, and $k_1[z] = k_1[y] = k_1(\xi)$.
  Then $(Y^*,z) \cdot x = 0$, and since $\xi$ is constrained over $k$, $(Y^*,z)$ cannot be free over $k$.
  Therefore, $z$ is algebraic over $k_1$, and so is $\xi$.
  Since $Y^*$ is generic over $k(\xi)$, $\xi$ is algebraic over $k$, and the extension $k_1(z)/k_1$ has the same separable and inseparable degrees as $k(\xi)/k$.

  Let $[x_t] = [y_t,1]$, for $t < D_s$, be the distinct $k$-conjugates of $\xi$, so $z_t = - Y^* \cdot y_t$ are the distinct $k_1$-conjugates of $z$.
  Then the irreducible polynomial of $z$ over $k_1$ is
  \begin{gather*}
    \fR
    = \prod_{t<D_s} \, (Z^* - z_t)^{D_i}
    = \prod_{t<D_s} \, (X^* \cdot x_t)^{D_i}.
  \end{gather*}
  It follows that $\fR \in k(X^*) \cap k^a[X^*] = k[X^*]$.
  In addition, being irreducible in $k(Y^*)[Z^*]$, and not having any factor involving $Y^*$ alone, it is irreducible in $k[X^*]$.

  On the one hand, $\fR$ is irreducible and $\fR(\lambda) = \prod \lambda(x_t) = 0$.
  On the other hand,
  \begin{gather*}
    n
    = \dim_{k(\xi)} \lambda
    \leq \dim_k \lambda.
  \end{gather*}
  Therefore, $\fR$ is the unique irreducible algebraic relation over $k$ satisfied by $\lambda$.
  In other words, $\xi$ is singly constrained over $k$ and $\fR^k_\xi \equiv \fR$.
  In addition, we saw that $\xi$ is $k(Y^*,z)$-rational.
  Since $\mu(X) = (Y^*,-z) \cdot X$ is another generic point of $V^*(\xi)$, $\xi$ is also $k(\lambda)$-rational.
\end{proof}

\begin{cor}
  \label{cor:SinglyConstrainedOne}
  Let $\xi_I \in \bP^{n_I}(L)$ be constrained over $k$.
  Then the following are equivalent:
  \begin{enumerate}
  \item The family $\xi_I$ is singly constrained.
  \item The family $\xi_I$ admits a least constrained sub-family.
  \item There exists $i \in I$ such that $\xi_{I \setminus \{i\}}$ is unconstrained.
  \end{enumerate}
\end{cor}
\begin{proof}
  \begin{cycprf}
  \item Let $J$ consist of all $i \in I$ such that $X^*_i$ occurs in $\fR^k_{\xi_I}$.
    Then it is clear that $\xi_J$ is the least constrained sub-family of $\xi_I$.
  \item If $\xi_J$ is the least constrained sub-family, then it is, in particular, constrained, so $J \neq \emptyset$, and then $\xi_{I \setminus \{i\}}$ is unconstrained for any $i \in J$.
  \item[\impfirst]
    By \autoref{thm:ResultantSplitting}.
  \end{cycprf}
\end{proof}

\begin{cor}
  \label{cor:SinglyConstrained}
  Let $\xi_I \in \bP(L)^I$, $i \in I$ and $J = I \setminus \{i\}$.
  Assume that $\xi_J$ is unconstrained over $k$, and let $\lambda_I \in V^*(\xi_I)$ be generic over $k(\xi_I)$.
  Then the following are equivalent:
  \begin{enumerate}
  \item The family $\xi_I$ is constrained over $k$.
  \item The family $\xi_I$ is singly constrained over $k$.
  \item The projective point $\xi_i$ is algebraic over $k(\lambda_J)$.
  \end{enumerate}
\end{cor}
\begin{proof}
  \begin{cycprf}
  \item[\impnnext]
    By \autoref{thm:ResultantSplitting}.
  \item[\impfirst]
    Since then $\dim_{k(\lambda_J)} \lambda_i = \dim_{k(\lambda_J,\xi_i)} \lambda_i = n_i < n_i+1$.
  \end{cycprf}
\end{proof}

\begin{cor}
  \label{cor:Unconstrained}
  A family of projective points $\xi_I \in \bP^{n_I}(L)$ is unconstrained over $k$ if and only if $\dim_k \xi_J \geq |J|$ for all $J \subseteq I$.
\end{cor}
\begin{proof}
  Let $\lambda_I \in V^*(\xi_I)$ be generic over $k(\xi_I)$.
  Then
  \begin{gather*}
    \dim_k \, \xi_I,\lambda_I
    = \dim_k \, \xi_I
    + \dim_{k(\xi_I)} \, \lambda_I
    = \dim_k \, \xi_I
    + \sum_{i \in I} n_i.
  \end{gather*}

  Assume first that $\xi_I$ is unconstrained, so $\lambda_I$ is algebraically free over $k$.
  Then
  \begin{gather*}
    \sum_{i \in I} \, (n_i+1)
    = \dim_k \, \lambda_I
    \leq \dim_k \, \xi_I, \lambda_I
    = \dim_k \, \xi_I + \sum_{i \in I} n_i.
  \end{gather*}
  Therefore, $\dim_k \, \xi_I \geq |I|$.
  The same logic applies to any sub-family $\xi_J$.

  For the converse, we may assume that $\xi_I$ is a minimal constrained family.
  Let $i \in I$, and let $J = I \setminus \{i\}$.
  By minimality, $\xi_J$ is unconstrained, and by \autoref{thm:ResultantSplitting}, $\xi_i$ is $k(\lambda_I)$-rational.
  In other words, the entire family $\xi_I$ is $k(\lambda_I)$-rational.
  Then
  \begin{gather*}
    \dim_k \, \lambda_I
    = \dim_k \, \xi_I,\lambda_I
    = \dim_k \, \xi_I
    + \sum_{i \in I} n_i.
  \end{gather*}
  On the other hand, $\dim_k \, \lambda_I < \sum_{i \in I} \, (n_i+1)$, so $\dim_k \, \xi_I < |I|$.
\end{proof}

\begin{lem}
  \label{lem:ResultantZeros}
  Assume that $\xi_I \in \bP(L)^I$ is singly constrained over $k$, and let $\mu_I$ be a family of linear forms in the corresponding indeterminates.
  Then $\fR^k_{\xi_I}(\mu_I) = 0$ if and only if there exists a specialisation $\zeta_I$ of $\xi_I$ over $k$ such that $\mu_I \in V^*(\zeta_I)$.
\end{lem}
\begin{proof}
  Let $\lambda_I \in V^*(\xi_I)$ be generic over $k(\xi_I)$.
  Then $\fR^k_{\xi_I}(\mu_I) = 0$ if and only if $\mu_I$ specialises $\lambda_I$ over $k$, if and only if there exists $\zeta_I$ such that $(\mu_I,\zeta_I)$ specialises $(\lambda_I,\xi_I)$ over $k$, if and only if there exists $\zeta_I$ specialising $\xi_I$ over $k$, such that $\mu_I \in V^*(\zeta_I)$.
\end{proof}

\begin{prp}
  \label{prp:ResultantMultiplicative}
  Let $\xi = [x],\upsilon = [y] \in \bP(K)$, and $\eta = \xi \otimes \upsilon = [x \otimes y]$.
  Let also $\zeta_I \in \bP(K)^I$ be unconstrained.
  Then the family $\eta,\zeta_I$ is (singly) constrained if and only if both $\xi,\zeta_I$ and $\upsilon,\zeta_I$ are.

  When this is the case, let $\fR^k_{\eta,\zeta_I}(W^*,Z^*_I)$, $\fR^k_{\xi,\zeta_I}(X^*,Z^*_I)$ and $\fR^k_{\upsilon,\zeta_I}(Y^*,Z^*_I)$ be the respective resultants.
  Then there exist multiplicities $m_\xi,m_\upsilon \geq 1$ such that (following \autoref{conv:ResultantEquivalence})
  \begin{gather*}
    \fR^k_{\eta,\zeta_I}(X^* \otimes Y^*,Z^*_I) \equiv \Bigl[ \fR^k_{\xi,\zeta_I}(X^*,Z^*_I) \Bigr]^{m_\xi} \Bigl[ \fR^k_{\upsilon,\zeta_I}(Y^*,Z^*_I) \Bigr]^{m_\upsilon}.
  \end{gather*}
\end{prp}
\begin{proof}
  Let $\lambda_I \in V^*(\zeta_I)$ be generic.
  By \autoref{cor:SinglyConstrained}, $\eta,\zeta_I$ is singly constrained over $k$ if and only if $\eta$ is algebraic over $k(\lambda_I)$, and similarly for $\xi$ et $\upsilon$.
  The first assertion follows.

  For the second assertion, let $\mu \in L[X]_1$, $\nu \in L[Y]_1$, and $\lambda_I \in \prod_{i \in I} L[Z_i]_1$ be arbitrary.

  Then $\fR^k_{\xi,\zeta_I}(\mu,\lambda_I) = 0$ if and only if there exist $\xi',\zeta'_I$ specialising $\xi,\zeta_I$ over $k$, such that $\mu \in V^*(\xi')$ and $\lambda_I \in V^*(\zeta_I)$.
  This can be extended to a specialisation over $k$ of $\xi,\upsilon,\zeta_I$ to $\xi',\upsilon',\zeta'_I$, i.e., of $\eta,\zeta_I$ to $\eta',\zeta_I$, where $\eta' = \xi' \otimes \upsilon'$.
  Consequently, $\fR^k_{\eta,\zeta_I}(\mu \otimes \nu,\lambda_I) = 0$.
  Similarly, if $\fR^k_{\upsilon,\zeta_I}(\nu,\lambda_I) = 0$.

  Conversely, if $\fR^k_{\eta,\zeta_I}(\mu \otimes \nu,\lambda_I) = 0$, then there exist $\eta',\zeta'_I$ specialising $\eta,\zeta_I$ over $k$, such that $\mu \otimes \nu \in V^*(\eta')$ and $\lambda_I \in V^*(\zeta_I)$.
  Such $\eta'$ is necessarily of the form $\xi' \otimes \upsilon'$, and $\xi',\upsilon',\zeta'_I$ specialises $\xi,\upsilon,\zeta_I$ over $k$.
  Now, either $\mu \in V^*(\xi')$ or $\nu \in V^*(\upsilon')$, so either $\fR^k_{\xi,\zeta_I}(\mu,\lambda_I) = 0$ or $\fR^k_{\upsilon,\zeta_I}(\nu,\lambda_I) = 0$.

  We have shown that $\fR^k_{\eta,\zeta_I}(X^* \otimes Y^*,Z^*_I)$ vanishes if and only if either $\fR^k_{\xi,\zeta_I}(X^*,Z^*_I)$ or $\fR^k_{\upsilon,\zeta_I}(Y^*,Z^*_I)$ does.
  Since the latter are irreducible, our assertion follows.
\end{proof}

\section{Normalised resultants}

\begin{conv}
  \label{conv:NormalisedResultantSetup}
  In this section we consider an intermediary field $k \subseteq K \subseteq L$, finitely generated over $k$.
  We say that $\xi \in \bP(K)$ is \emph{ample} (for the extension $K/k$) if $K = k(\xi)$.
  A family $\xi_I \in \bP(K)^I$ is ample if each $\xi_i$ is.
\end{conv}

\begin{lem}
  \label{lem:NormalisedResultantSetup}
  Let $K/k$ be a finitely generated extension.
  Let $\xi_I \in \bP(K)^I$, and let $\lambda_I \in V^*(\xi_I)$ be generic.
  \begin{enumerate}
  \item
    \label{item:NormalisedResultantSetupUnconstrained}
    If $\xi_I$ is unconstrained, then $\dim_{k(\lambda_I)} K(\lambda_I) + |I| = \dim_k K$.
    In particular, $|I| \leq \dim_k K$, and if $|I| = \dim_k K$, then $K(\lambda_I) / k(\lambda_I)$ is a finite algebraic extension.
  \item
    \label{item:NormalisedResultantSetupConstrained}
    If $|I| = \dim_k K + 1$, then $\xi_I$ is necessarily constrained, and if $|I| > \dim_k K + 1$, then $\xi_I$ is necessarily multiply constrained.
  \item
    \label{item:NormalisedResultantSetupMultiplicity}
    If $|I| = \dim_k K + 1$ and $\xi_I$ is singly constrained, then $K(\lambda_I) / k(\lambda_I)$ is an algebraic extension.
  \item
    \label{item:NormalisedResultantSetupMultiplicityOne}
    If $|I| = \dim_k K + 1$, and $\xi_I$ is ample, then $K(\lambda_I) = k(\lambda_I)$.
    In fact, it is enough to assume that there exists $i \in I$ such that $\xi_i$ is ample, and $\xi_{I \setminus\{i\}}$ is unconstrained.
  \end{enumerate}
\end{lem}
\begin{proof}
  For \autoref{item:NormalisedResultantSetupUnconstrained}, recall that the family $\lambda_I$ is generic over $k$, and each $\lambda_i$ satisfies an algebraic (linear, even) relation with $K$.
  Item \autoref{item:NormalisedResultantSetupConstrained} follows.

  For \autoref{item:NormalisedResultantSetupMultiplicity}, there exists $i \in I$ such that $\xi_{I \setminus \{i\}}$ is unconstrained, and we may apply \autoref{item:NormalisedResultantSetupUnconstrained}.
  Finally, let us assume that the hypotheses of \autoref{item:NormalisedResultantSetupMultiplicityOne} hold.
  Then, $K = k(\xi_i)$ by hypothesis, and by \autoref{thm:ResultantSplitting}, $\xi_i$ is $k(\lambda_I)$-rational.
\end{proof}

\begin{dfn}
  \label{dfn:NormalisedDegree}
  Let $|I| = \dim_k K$ and let $\xi_I \in \bP(K)^I$.
  We define $\deg^{K/k} \xi_I \in \bN$, the \emph{degree} of $\xi_I$ relative to the extension $K/k$, as follows:
  \begin{itemize}
  \item If $\xi_I$ is unconstrained, then we let $\lambda_I \in V^*(\xi_I)$ be generic, and define
    \begin{gather*}
      \deg^{K/k} \xi_I = \bigl[ K(\lambda_I) : k(\lambda_I) \bigr].
    \end{gather*}
  \item If $\xi_I$ is constrained, then $\deg^{K/k} \xi_I = 0$.
  \end{itemize}
\end{dfn}

\begin{dfn}
  \label{dfn:NormalisedResultant}
  Let $|I| = \dim_k K + 1$ and let $\xi_I \in \bP^{n_I}(K)$.
  We define $\fR^{K/k}_{\xi_I} \in k[X^*_I]$, the \emph{normalised resultant} associated to $\xi_I$ relative to the extension $K/k$, as follows:
  \begin{itemize}
  \item If $\xi_I$ is singly constrained, then we let $\lambda_I \in V^*(\xi_I)$ be generic, and define
    \begin{gather*}
      \fR^{K/k}_{\xi_I} = \bigl( \fR^k_{\xi_I} \bigr)^{\bigl[ K(\lambda_I) : k(\lambda_I) \bigr]}.
    \end{gather*}
  \item If $\xi_I$ is multiply constrained, then $\fR^{K/k}_{\xi_I} = 1$ (or any other non-zero constant).
  \end{itemize}
\end{dfn}

\begin{rmk}
  \label{rmk:NormalisedResultantReduction}
  As in \autoref{rmk:ResultantReduction}, let $J \subseteq I$ and $I_0 = I \setminus J$.
  Assume that $\xi_I \in \bP^{n_I}(K)$, and that $\xi_J$ is unconstrained.
  Let $\lambda_I \in V^*(\xi_I)$ be generic, and let $k_1 = k(\lambda_J)$.
  Let also $K_1 = K(\lambda_J)$.

  Then $\dim_{k_1} K_1 + |J| = \dim_k K$, by \autoref{lem:NormalisedResultantSetup}\autoref{item:NormalisedResultantSetupUnconstrained}, and
  \begin{gather}
    \label{eq:NormalisedResultantReduction}
    \bigl[ K(\lambda_I) : k(\lambda_I) \bigr] = \bigl[ K_1(\lambda_{I_0}) : k_1(\lambda_{I_0}) \bigr].
  \end{gather}

  \begin{enumerate}
  \item If $|I| = \dim_k K$ and $\xi_I$ is unconstrained over $k$, then $|I_0| = \dim_{k_1} K_1$, $\xi_{I_0}$ is unconstrained over $k_1$, and by \autoref{eq:NormalisedResultantReduction},
    \begin{gather*}
      \deg^{K/k} \xi_I = \deg^{K_1/k_1} \xi_{I_0}.
    \end{gather*}
  \item Assume now that $|I| = \dim_k K + 1$ and $\xi_I$ is singly constrained over $k$.
    Then $|I_0| = \dim_{k_1} K_1 + 1$, $\xi_{I_0}$ is singly constrained over $k_1$, and $\fR^k_{\xi_I} \equiv \fR^{k_1}_{\xi_{I_0}}$, under the identification of \autoref{rmk:ResultantReduction}\autoref{item:ResultantReductionIsomorphism}.
    Again, using \autoref{eq:NormalisedResultantReduction},
    \begin{gather*}
      \fR^{K/k}_{\xi_I} \equiv \fR^{K_1/k_1}_{\xi_{I_0}}.
    \end{gather*}
  \end{enumerate}
\end{rmk}

\begin{lem}
  \label{lem:NormalisedDegree}
  Let $|J| = \dim_k K$.
  Then for every $\xi \in \bP(K)$ and $\zeta_J \in \bP(K)^J$:
  \begin{gather*}
    \deg^{K/k} \zeta_J = \deg_{X^*} \fR^{K/k}_{\xi,\zeta_J}.
  \end{gather*}
\end{lem}
\begin{proof}
  Assume first that $\zeta_J$ is unconstrained.
  By \autoref{rmk:NormalisedResultantReduction} we may assume that $J = \emptyset$.
  Then $K/k$ is a finite algebraic extension, and $\deg^{K/k} \emptyset = [K:k]$.
  By \autoref{thm:ResultantSplitting}, $\bigl[ k(\xi) : k \bigr] = \deg_{X^*} \fR^k_\xi$.
  Let $\lambda \in V^*(\xi)$ be generic, and let $\lambda'$ consist of all the coefficients of $\lambda$ bar one, corresponding to a non-zero homogeneous coordinate of $\xi$.
  Then $\lambda'$ is generic over $K$, and $\xi$ is $k(\lambda)$-rational by \autoref{thm:ResultantSplitting}, so $k(\lambda) = k(\lambda,\xi) = k(\lambda',\xi)$.
  Therefore
  \begin{gather*}
    \bigl[ K : k(\xi) \bigr]
    = \bigl[ K(\lambda') : k(\lambda',\xi) \bigr]
    = \bigl[ K(\lambda) : k(\lambda) \bigr],
  \end{gather*}
  and
  \begin{gather*}
    [K : k]
    = \bigl[ K : k(\xi) \bigr]  \bigl[ k(\xi) : k \bigr]
    = \bigl[ K(\lambda) : k(\lambda) \bigr] \deg_{X^*} \fR^k_\xi = \deg_{X^*} \fR^{K/k}_\xi.
  \end{gather*}

  In the opposite case, $\zeta_J$ is constrained.
  If $\xi,\zeta_J$ is multiply constrained, then $\fR^{K/k}_{\xi,\zeta_J}$ is constant, and if $\xi,\zeta_J$ is singly constrained, then so must be $\zeta_J$, and $\fR^k_{\xi,\zeta_J} \equiv \fR^k_{\zeta_J}$.
  Either way, $\deg_{X^*} \fR^{K/k}_{\xi,\zeta_J} = 0$.
\end{proof}

\begin{thm}
  \label{thm:NormalisedResultant}
  Assume that $K/k$ is a finitely generated extension, and let $|I| = \dim_k K + 1$.
  \begin{enumerate}
  \item
    \label{item:NormalisedResultantAmple}
    If $\xi_I \in \bP(K)^I$ is ample, then $\fR^{K/k}_{\xi_I} \equiv \fR^k_{\xi_I}$.
  \item
    \label{item:NormalisedResultantMultiplicative}
    Let $i \in I$ and $J = I \setminus \{i\}$.
    Then for every $\xi,\upsilon \in \bP(K)$, $\eta = \xi \otimes \upsilon$ and $\zeta_J \in \bP(K)^J$:
    \begin{gather}
      \label{eq:NormalisedResultantMultiplicative}
      \fR^{K/k}_{\eta,\zeta_J}(X^* \otimes Y^*,Z^*_J)
      \equiv
      \fR^{K/k}_{\xi,\zeta_J}(X^*,Z^*_J)
      \cdot \fR^{K/k}_{\upsilon,\zeta_J}(Y^*,Z^*_J).
    \end{gather}
  \end{enumerate}
  Moreover, these properties determine the map $\xi_I \mapsto \fR^{K/k}_{\xi_I}$ (up to equivalence).
\end{thm}
\begin{proof}
  Let us start with the moreover part.
  If $\xi_I$ is ample, then $\fR^{K/k}_{\xi_I}$ is determined by \autoref{item:NormalisedResultantAmple}.
  If $\xi_I$ is not ample, then there exists $i \in I$ such that $\xi_i$ is not ample.
  Let $J = I \setminus \{i\}$ and let $\tau$ be ample.
  Then
  \begin{gather*}
    \fR^{K/k}_{\xi_i \otimes \tau,\xi_J}(X^*_i \otimes T^*,X^*_J)
    \equiv
    \fR^{K/k}_{\xi_I}(X^*_I)
    \cdot \fR^{K/k}_{\tau,\xi_J}(T^*,X^*_J).
  \end{gather*}
  Both $\tau$ and $\xi_i \otimes \tau$ are ample, so by induction on the number of non-ample members, $\fR^{K/k}_{\xi_I}$ is determined for every $\xi_I$.

  Let us now prove the main assertion.
  For \autoref{item:NormalisedResultantAmple}, just apply \autoref{lem:NormalisedResultantSetup}\autoref{item:NormalisedResultantSetupMultiplicityOne}.
  For \autoref{item:NormalisedResultantMultiplicative}, we consider several cases.

  Assume first that $\zeta_J$ is unconstrained.
  The resultant $\fR^{K/k}_{\eta,\zeta_J}(W^*,Z^*_J)$ is homogeneous in $W^*$, of degree $\deg^{K/k} \zeta_J$ by \autoref{lem:NormalisedDegree}.
  By \autoref{prp:ResultantMultiplicative}, there exist multiplicities $m,m'$ such that
  \begin{gather*}
    \fR^{K/k}_{\eta,\zeta_J}(X^* \otimes Y^*,Z^*_J)
    \equiv
    \Bigl[ \fR^k_{\xi,\zeta_J}(X^*,Z^*_J) \Bigr]^m
    \Bigl[ \fR^k_{\upsilon,\zeta_J}(Y^*,Z^*_J) \Bigr]^{m'}.
  \end{gather*}
  The left hand side is homogeneous in $X^*$ and in $Y^*$, of degree $\deg^{K/k} \zeta_J$.
  Using \autoref{lem:NormalisedDegree} again on the right hand side and comparing degrees, we obtain \autoref{eq:NormalisedResultantMultiplicative}.

  Assume next that $\zeta_J$ is singly constrained.
  In this case, $\fR^k_{\zeta_J}(Z^*_J)$ is an irreducible polynomial, and all the factors occurring in \autoref{eq:NormalisedResultantMultiplicative} are its powers (possibly constants).
  Choose $j \in J$ such that $Z^*_j$ occurs in $\fR^k_{\zeta_J}(Z^*_J)$, i.e., such that $\zeta_{J \setminus \{j\}}$ is unconstrained.
  Then it will suffice to show that
  \begin{gather*}
    \deg_{Z^*_j} \fR^{K/k}_{\eta,\zeta_J}
    =
    \deg_{Z^*_j} \fR^{K/k}_{\xi,\zeta_J}
    +
    \deg_{Z^*_j} \fR^{K/k}_{\upsilon,\zeta_J}.
  \end{gather*}
  By \autoref{rmk:NormalisedResultantReduction} we may assume that $I = \{i,j\}$ and $\zeta_J = \zeta$.
  We are therefore left with proving
  \begin{gather}
    \label{eq:NormalisedResultantMultiplicativeSinglyConstrained}
    \deg_{T^*} \fR^{K/k}_{\eta,\tau}
    =
    \deg_{T^*} \fR^{K/k}_{\xi,\tau}
    +
    \deg_{T^*} \fR^{K/k}_{\upsilon,\tau}
  \end{gather}
  with $\tau = \zeta$.
  By \autoref{lem:NormalisedDegree}, \autoref{eq:NormalisedResultantMultiplicativeSinglyConstrained} is equivalent to
  \begin{gather*}
    \deg^{K/k} \eta
    =
    \deg^{K/k} \xi
    +
    \deg^{K/k} \upsilon.
  \end{gather*}
  Therefore, it does not depend on $\tau$.
  When $\tau$ is ample, \autoref{eq:NormalisedResultantMultiplicativeSinglyConstrained} holds by the case already proved.

  The last case to consider is when $\zeta_J$ is multiply constrained.
  Then every family extending it is multiply constrained, and \autoref{eq:NormalisedResultantMultiplicative} holds as $1 \equiv 1 \cdot 1$.
\end{proof}

\begin{cor}
  \label{cor:NormalisedDegree}
  Assume that $K/k$ is a finitely generated extension, and let $|I| = \dim_k K$.
  \begin{enumerate}
  \item
    \label{item:NormalisedDegreeAmple}
    If $\xi_I \in \bP(K)^I$ is ample and $\upsilon \in \bP(K)$ arbitrary, then
    \begin{gather*}
      \deg^{K/k} \xi_I = \deg_{Y^*} \fR^k_{\upsilon,\xi_I}(Y^*,X^*_I).
    \end{gather*}
  \item
    \label{item:NormalisedDegreeAdditive}
    If $J = I \setminus \{i\}$, $\xi,\upsilon \in \bP(K)$, $\zeta_J \in \bP(K)^J$, then
    \begin{gather*}
      \deg^{K/k}(\xi \otimes \upsilon,\zeta_J) = \deg^{K/k}(\xi,\zeta_J) + \deg^{K/k}(\upsilon,\zeta_J).
    \end{gather*}
  \end{enumerate}
  Moreover, $\deg^{K/k}$ is determined by the restriction of \autoref{item:NormalisedDegreeAmple} to the case where $\upsilon$ is also ample together with \autoref{item:NormalisedDegreeAdditive}.
\end{cor}
\begin{proof}
  Apply \autoref{lem:NormalisedDegree} to \autoref{thm:NormalisedResultant}.
\end{proof}

\begin{rmk}
  \label{rmk:NormalisedDegreeIntersection}
  In the setting of \autoref{cor:NormalisedDegree}, let $\xi_I \in \bP(K)^I$ and let $\upsilon \in \bP(K)$ be ample.
  Define $\hat{\upsilon} = \upsilon \otimes \bigotimes_i \xi_i$.
  Finally, assume that $k$ is relatively algebraically closed in $K$, and let $X$ be the projective locus of $\hat{\upsilon}$ over $k$.

  We may then naturally identify each $\xi_i$ with an effective divisor on $X$, and $\deg^{K/k} \xi_I$ is their intersection number.
\end{rmk}

\section{The wedge operation on resultants}
\label{sec:ResultantWedge}

Recall the notions of a splittable polynomial and the wedge operation from \autoref{dfn:Splittable} and \autoref{dfn:Wedge}.

\begin{ntn}
  \label{ntn:MultiWedge}
  Let $A$ be a ring and $R \in A[X^*_I]$ a polynomial that is splittable in each $X^*_i$.
  If $f \in A[X_i]$ is homogeneous, we shall denote the wedge of $f$ and $R$ relative to the duality between $X_i$ and $X^*_i$ by
  \begin{gather*}
    f \wedge_i R \in A[X^*_{I \setminus \{i\}}].
  \end{gather*}
  When there is no risk of ambiguity, i.e., when the polynomial $f$ clearly determines $i$, we may allow ourselves to drop it and write $f \wedge R$.
\end{ntn}

Let $L/k$ be as in \autoref{sec:Resultant}.
By virtue of \autoref{thm:ResultantSplitting}, if $\xi_I \in \bP(L)^I$ is singly constrained, $i \in I$, and $f \in A[X_i]$ is homogeneous for some $k$-algebra $A$, then $\fR^k_{\xi_I}$ is splittable as a polynomial in $X^*_i$, and the wedge product $f \wedge \fR^k_{\xi_I} \in A[X^*_{I \setminus \{i\}}]$ is defined.
Notice that this includes the case (not covered by \autoref{thm:ResultantSplitting}) where $\xi_{I \setminus \{i\}}$ is (singly) constrained, so $X^*_i$ does not occur in $\fR^k_{\xi_I}$ and $f \wedge \fR^k_{\xi_I} = \bigl(\fR^k_{\xi_I}\bigr)^{\deg f}$.

For a point $\xi = [x] \in \bP^n(L)$ and $d \in \bN$, let $x^{\otimes d} = (x^\alpha : |\alpha| = d)$, namely the sequence of all monomials of degree $d$ evaluated at $x$, and $\xi^{\otimes d} = [x^{\otimes d}]$.

\begin{lem}
  \label{lem:ResultantWedgeVeronese}
  Let $L/k$ be a field extension, and let $\xi_I \in \bP^{n_I}(L)$ be singly constrained.
  Let $i \in I$ and $J = I \setminus \{i\}$, and assume that $\xi_J$ is unconstrained.
  Let $d_i \geq 1$, let $F_i = \sum_{|\alpha| = d_i} T^*_\alpha X_i^\alpha$ be an indeterminate polynomial of degree $d_i$, and let $\zeta_i = \xi_i^{\otimes d_i}$.

  Then $\zeta_i,\xi_J$ is a singly constrained family, and under the natural identification between polynomials of degree $d_i$ in $X$ and linear forms on $X^{\otimes d}$ we have
  \begin{gather*}
    F_i \wedge \fR^k_{\xi_I} \equiv \fR^k_{\zeta_i,\xi_J} \in k[T^*_i,X^*_J].
  \end{gather*}
\end{lem}
\begin{proof}
  By \autoref{rmk:ResultantReduction}, we may assume that $I = \{i\}$ and drop the index $i$.
  Since $\xi$ is algebraic (equivalently, a constrained singleton family) over $k$, so is $\zeta = \xi^{\otimes d}$.
  The resultant $\fR^k_\xi$ factors as $\prod_{t<D} \, (X^* \cdot x_t)$, where $\xi_t = [x_t]$ are the $k$-conjugates of $\xi$, and $F \wedge \fR^k_\xi = \prod_{t<D} F(x_t)$.
  On the other hand, $\xi_t^{\otimes d}$ are the conjugates of $\zeta$, and each distinct conjugate is repeated with multiplicity $D_i$, where $D_i$ is the inseparable degree of $k(\xi) = k(\zeta)$ over $k$.
  Therefore $\fR^k_\zeta = \prod_{t<D} \, (F \cdot x_t^{\otimes d}) = \prod_{t<D} F(x_t)$ up to a factor in $k^\times$.
\end{proof}

\begin{lem}
  \label{lem:ResultantWedgeCommute}
  Let $L/k$ be a field extension, let $\xi_I \in \bP^{n_I}(L)$ be singly constrained, and let $i,j \in I$ be distinct.

  Then for any two homogeneous polynomials $f_i \in L[X_i]_{d_i}$ and $f_j \in L[X_j]_{d_j}$, we have
  \begin{gather*}
    f_j \wedge f_i \wedge \fR^k_{\xi_I} = f_i \wedge f_j \wedge \fR^k_{\xi_I}.
  \end{gather*}
  Equivalently, the same holds for any two indeterminate homogeneous polynomials in $X_i$ and in $X_j$.
\end{lem}
\begin{proof}
  Let $\fR = \fR^k_{\xi_I}$, $D_i = \deg_{X^*_i} \fR$ and $D_j = \deg_{X^*_j} \fR$.
  If $D_i = 0$, then
  \begin{gather*}
    f_j \wedge f_i \wedge \fR
    = \bigl( f_j \wedge \fR \bigr)^{d_i}
    = f_i \wedge f_j \wedge \fR.
  \end{gather*}
  If $d_i = 0$, then $\deg_{X^*_i} f_j \wedge \fR = D_id_j$, and
  \begin{gather*}
    f_j \wedge f_i \wedge \fR
    = f_i^{D_id_j}
    = f_i \wedge f_j \wedge \fR.
  \end{gather*}
  Similarly, if $D_j = 0$ or $d_j = 0$.

  We are left with the case where none of the degrees vanishes.
  In particular, $\xi_{I \setminus \{i,j\}}$ is unconstrained, and by \autoref{rmk:ResultantReduction} we may assume that $I = \{i,j\}$.
  Let $F_i$ be an indeterminate homogeneous polynomial in $X_i$ of degree $d_i$ and $\zeta_i = \xi_i^{\otimes d_i}$, and similarly for $F_j$ and $\zeta_j$.
  Then, by \autoref{lem:ResultantWedgeVeronese}, both $F_j \wedge F_i \wedge \fR$ and $F_i \wedge F_j \wedge \fR$ are instances of $\fR^k_{\zeta_i,\zeta_j}$, and can only differ by a factor $\alpha \in k^\times$.
  Substituting $(X^*_i \cdot X)^{d_i}$ for $F_i$ and $(X^*_j \cdot X_j)^{d_j}$ for $F_j$, either expression specialises to $\fR^{d_id_j}$, so $\alpha = 1$.
\end{proof}

\section{A characterisation of resultants}
\label{sec:ResultantAbstract}

For what follows, we consider a finite set $I$ and $n_I \in \bN^I$.
For each $i \in I$, we let $X_i$ be a tuple of $n_i+1$ indeterminates, and let $X^*_i$ be the dual indeterminates.

\begin{dfn}
  \label{dfn:ResultantAbstract}
  Let $A$ be a ring and $\fR \in A[X^*_I]$.
  It is a \emph{resultant} in $X^*_I$ over $A$ if for every $i \in I$:
  \begin{enumerate}
  \item
    \label{item:ResultantAbstractSplittable}
    As a polynomial in $X^*_i$, $\fR$ is splittable.
  \item
    \label{item:ResultantAbstractInduction}
    For every $A$-algebra $B$, and every homogeneous $f \in B[X_i]$, the polynomial $f \wedge \fR$ is a resultant in $X^*_{I \setminus \{i\}}$ over $B$.
  \item
    \label{item:ResultantAbstractCommute}
    In addition, if $j \in I \setminus \{i\}$ and $g \in B[X_j]$ is homogeneous, then
    \begin{gather*}
      g \wedge f \wedge \fR = f \wedge g \wedge \fR.
    \end{gather*}
  \end{enumerate}
\end{dfn}

\begin{dfn}
  \label{dfn:WeakResultantAbstract}
  Let $A$ be a ring and $\fR \in A[X^*_I]$.
  It is a \emph{weak resultant} over $A$ in $X^*_I$ if for every $i \in I$:
  \begin{enumerate}
  \item
    As a polynomial in $X^*_i$, $\fR$ is splittable.
  \item
    For every $j \in I \setminus \{i\}$, the polynomial $F \wedge \fR \in B[X^*_J]$ is splittable as a polynomial in $X^*_j$, where $F = \sum_{|\alpha|=n_j} T^*_\alpha X_i^\alpha$ is an indeterminate homogeneous polynomial in $X_i$, of degree $n_j$, and $B = A[T^*]$.
  \end{enumerate}
\end{dfn}

\begin{rmk}
  \label{rmk:ResultantAbstract}
  The following are immediate:
  \begin{enumerate}
  \item Every resultant is, in particular, a weak resultant.
  \item For fixed $I$ and $n_I$, being a (weak) resultant is defined by a family of homogeneous constraints (over $\bZ$) on the coefficients.
    In particular, it is preserved under ring morphisms.
  \item The zero polynomial is always a resultant; if $I = \emptyset$, then every $a \in A$ is a resultant over $A$; and if $I$ is a singleton, then every splittable polynomial in $A[X^*]$ is a resultant.
  \item If $J \subseteq I$ and $\fR$ is a resultant in $X^*_I$ over $A$, then it is also a resultant in $X^*_{I \setminus J}$ over $A[X^*_J]$.
  \item If $J \subseteq I$ and $\fR \in A[X^*_J] \subseteq A[X^*_I]$, then $\fR$ is a resultant in $X^*_I$ if and only if it is one in $X^*_J$.
  \item The product of any two (weak) resultants over $A$ in $X^*_I$ (or, by the previous item, in sub-families thereof) is again a (weak) resultant.
  \end{enumerate}
\end{rmk}

\begin{exm}
  \label{exm:ResultantAbstractWeak}
  Let us consider the elementary symmetric polynomials in projective dimension $n = 2$, in $D = 2$ indeterminates.
  A transcendence degree calculation reveals that the coefficients of $G_2(X^*)$ must satisfy a single homogeneous relation over $\bZ$.
  The precise relation is not very important, but let us say it is of degree $d$.
  Let $A = \bZ[\varepsilon] = \bZ[S]/(S^{2d+1})$.

  Let $m$ denote a projective dimension that may vary, so let us introduce indeterminates $X = (X_i : i \in \bN)$ and similarly $Y$ and dual indeterminates $X^*$ and $Y^*$.
  Let $R_m = \sum_{i\leq m} X^*_i Y^*_i \in \bZ[X^*,Y^*]$, and let $f_{k,m}(X) = \sum_{\alpha \in \bN^{m+1}, \ |\alpha|=k} T_\alpha X^\alpha$ be the indeterminate polynomial of degree $k$ in projective dimension $m$.
  Then $R_m$ is splittable in either $X^*$ or $Y^*$, being already linear, and $f_{k,m}(X) \wedge R_m = f_{k,m}(Y^*)$.
  Let us denote $f_{m,m}$ by $f_m$, for short.

  On the one hand, $f_2(Y^*)$ is not splittable, and, since $\varepsilon^{2d} \neq 0$, neither is $\varepsilon^2 f_2(Y^*) = f_2 \wedge (\varepsilon R_2)$.
  On the other hand, $f_{2d+1} \wedge (\varepsilon R_{2d+1}) = \varepsilon^{2d+1} f_{2d+1} \wedge R_{2d+1} = 0$ is splittable.
  Therefore, there exists a least $m \geq 3$ such that $f_m \wedge (\varepsilon R_m)$ is splittable.
  Then $f_{m-1} \wedge (\varepsilon R_{m-1})$ is not splittable, and \textit{a fortiori}, $f_{m-1,m} \wedge (\varepsilon R_m)$ is not.

  We conclude that $\varepsilon R_m$ is a weak resultant but not a resultant.
\end{exm}

On the face of it, \autoref{exm:ResultantAbstractWeak} makes a non-trivial use of the nilpotent element $\varepsilon$, and indeed, we are going to show that over a reduced ring, weak resultants and resultants agree.
Since both notions are defined by polynomial equations, it will suffice to show this over a field.
More precisely, we will show that a polynomial over a field $k$ is a (weak) resultant if and only if its irreducible factors are of the form $\fR^k_{\xi_I}$, as defined in \autoref{dfn:Resultant}.

\begin{lem}
  \label{lem:ResultantAbstractSplittable}
  Let $k$ be a field and $L/k$ an algebraically closed extension of infinite transcendence degree.
  Let $R_0, R \in k[X^*_I]$ be non-zero polynomials.
  Assume that $R_0$ is an irreducible factor of $R$, and let $\lambda_I$ be a generic root of $R_0$ with coefficients in $L$.
  Finally, let $i \in I$, assume that $R$ is splittable in $X^*_i$, and let $J = I \setminus \{i\}$.
  \begin{itemize}
  \item If $X^*_i$ does not occur in $R_0$, then $R(X^*_i,\lambda_J) = R_0(X^*_i,\lambda_J) = 0$.
  \item If $X^*_i$ occurs in $R_0$, then $\lambda_J$ is free over $k$, and there exists a unique $\xi_i = [x_i] \in \bP^{n_i}(L)$ such that $X^*_i \cdot x_i \mid R(X^*_i,\lambda_J)$ and $\lambda_i(x_i) = 0$.
    Moreover, $X^*_i \cdot x_i \mid R_0(X^*_i,\lambda_J)$.
  \end{itemize}
\end{lem}
\begin{proof}
  The first case is clear, so we consider the second.
  Since $R_0$ is the unique algebraic relation over $k$ satisfied by $\lambda_I$, and $X^*_i$ occurs there, the remaining $\lambda_J$ satisfy no algebraic relation over $k$.
  In particular, $R(X^*_i,\lambda_J) \neq 0$, and by hypothesis, it splits as
  \begin{gather*}
    R(X^*_i,\lambda_J) = \prod_{t<d} X^*_i \cdot x_{i,t},
  \end{gather*}
  where $\xi_{i,t} = [x_{i,t}] \in \bP^{n_i}(L)$.
  This factorisation is unique up to permutation of the family $\Xi = (\xi_{i,t} : t < d)$.

  Since $R(\lambda_i,\lambda_J) = 0$, there exists $t < d$ such that $\lambda_i(x_{i,t}) = 0$.
  Since $\lambda_i$ is a generic root of $R(X^*_i,\lambda_J)$, this determines $\xi = \xi_{i,t}$ uniquely (although it may appear more than once in $\Xi$).
  The moreover part holds since $R_0(\lambda_i,\lambda_J) = 0$.
\end{proof}

\begin{lem}
  \label{lem:ResultantAbstractAlgebraic}
  Let $k$ be a field and $L/k$ an algebraically closed extension of infinite transcendence degree.
  Let $X$ and $Y$ represent homogeneous coordinates in projective dimensions $n$ and $m$, respectively.
  Let $\fR \in k[X^*,Y^*]$ be a non-zero weak resultant, and let $\fR_0$ be an irreducible factor of $\fR$ in which both $X^*$ and $Y^*$ occur.
  Let $\lambda,\mu$ be a generic root of $\fR_0(X^*,Y^*)$, with coefficients in $L$.

  Finally, let $\xi = [x] \in \bP^n(L)$ be the unique point, as per \autoref{lem:ResultantAbstractSplittable}, such that $\lambda(x) = 0$ and $X^* \cdot x \mid \fR_0(X^*,\mu)$.
  Then $\xi$ is algebraic over $k(\lambda)$.
\end{lem}
\begin{proof}
  Since $(\lambda,\mu)$ is a generic root of $\fR_0$, each of $\lambda$ and $\mu$, separately, is entirely generic over $k$.
  In addition, $\lambda$ is a generic root of $\fR_0(X^*,\mu)$, of which $X^* \cdot x$ is a factor, so $\lambda$ is generic in $V^*(\xi)$, and $\xi$ cannot be algebraic over $k$.
  Fix a generic $f \in V^*_m(\xi)$, i.e., $f \in L[X]_m$ that vanishes at $\xi$ and is generic such.
  Since $\xi$ is not algebraic over $k$, $f$ is a generic polynomial over $k$, i.e., $\dim_k f = N = \binom{n+m}{m}$.

  Since $f$ is generic over $k$, $\fS(Y^*) = f \wedge \fR(X^*,Y^*)$ is non-zero.
  Since $\fR_0$ is a factor of $\fR$, it is also splittable in $X^*$, and $\fS_0 = f \wedge \fR_0$ is a factor of $\fS$.
  On the other hand $X^* \cdot x \mid \fR_0(X^*,\mu)$ and $f(x) = 0$, so $\fS_0(\mu) = f \wedge \fR_0(X^*,\mu) = 0$.

  Since $\fR$ is a weak resultant, $\fS$ is splittable in $Y^*$, and therefore so is $\fS_0$.
  Since $\fS_0(\mu) = 0$, $\fS_0$ admits a factor $Y^* \cdot y$ such that $\mu(y) = 0$.
  We may assume that at least one coordinate of $y$ equals one.
  The point $\xi$ is algebraic over $k(\mu)$, since $X^* \cdot x$ is one of finitely many factors of $\fR_0(X^*,\mu)$, and similarly, $y$ is algebraic over $k(f)$.

  Since $f$ is generic vanishing at $\xi$ it satisfies a unique algebraic relation over $k$ with $\xi$.
  Since $\xi$ is algebraic over $k(\mu)$, $f$ satisfies with $\mu$ a unique algebraic relation over $k$.

  Since $\mu(y) = 0$, we have
  \begin{gather*}
    m
    = \dim_{k(f)} \mu
    = \dim_{k(f,y)} \mu
    \leq \dim_{k(y)} \mu
    \leq m.
  \end{gather*}
  Consequently, $\mu$ is generic vanishing at $y$.
  Similarly, $f$ and $\mu$ are algebraically independent over $k(y)$, and \textit{a fortiori} so are $f$ and $\xi$.

  Let $W \subseteq \bP^n(L)$ be the projective locus of $\xi$ over $k(y)^a$.
  Let $E \subseteq L^N = L[X]_m$ be the affine locus of $f$ over $k(y)^a$.
  Since $f$ and $\xi$ are algebraically independent over $k(y)$, the pair $(f,\xi)$ is generic in $E \times W$, so every member of $E$ vanishes on $W$.
  Therefore,
  \begin{gather*}
    \dim_{k(y)} f \leq \lindim_L E \leq \lindim_L \bigl\{ g \in L[X]_m : g \ \text{vanishes on} \ W \bigr\} = N - \lindim_L L[W]_m.
  \end{gather*}
  If $W$ is not the singleton $\{\xi\}$, then $\lindim_L L[W]_m \geq m+1$, which is impossible, since
  \begin{gather*}
    \dim_{k(y)} f \geq \dim_k f - \dim_k y \geq N - m.
  \end{gather*}
  Therefore $W = \{\xi\}$, so $\xi$ is algebraic over $k(y)$.

  We have already observed that $\xi$ is algebraic over $k(\mu)$, so
  \begin{gather*}
    m + 1
    = \dim_k \mu
    = \dim_k \mu,\xi
    = \dim_k \xi + \dim_{k(\xi)} \mu.
  \end{gather*}
  On the other hand, since $\xi$ is also algebraic over $k(y)$, we have
  \begin{gather*}
    m
    = \dim_{k(y)} \mu
    = \dim_{k(y,\xi)} \mu
    \leq \dim_{k(\xi)} \mu.
  \end{gather*}
  Therefore, $\dim_k \xi \leq 1$.

  Finally, $\lambda$ is generic over $k$, and vanishes at $\xi$.
  Therefore $\dim_k \xi = 1$ and $\dim_{k(\lambda)} \xi = 0$.
\end{proof}

\begin{lem}
  \label{lem:ResultantAbstract}
  Let $k$ be a field and $L/k$ an algebraically closed extension of infinite transcendence degree.
  Let $\fR \in k[X^*_I]$ be a non-zero weak resultant.
  Then every irreducible factor of $\fR$ is of the form $\fR^k_{\xi_{I_0}}$ where $I_0 \subseteq I$ and $\xi_{I_0}$ is a minimally constrained family over $k$.
\end{lem}
\begin{proof}
  Let $\fR_0$ be an irreducible factor of $\fR$, and let $I_0$ be the set of $i \in I$ such that $X^*_i$ occurs in $\fR_0$.
  Let $\lambda_I$ be a generic root of $\fR_0$.
  For $i \in I_0$, let $\xi_i = [x_i]$ be as per \autoref{lem:ResultantAbstractSplittable}.

  Let $i \in I_0$ and $J = I_0 \setminus \{i\}$.
  We have $\fR_0 \in k[X^*_i,X^*_J]$, and $X^*_i \cdot x_i \mid \fR_0(X^*_i,\lambda_J)$, so $\xi_i$ is algebraic over $k(\lambda_J)$.
  If $j \in J$, then, applying \autoref{lem:ResultantAbstractAlgebraic} to $\fR_0$ and $\fR$ over $k(\lambda_{J \setminus \{j\}})$, we see that $\xi_j$ is also algebraic over $k(\lambda_J)$.
  Therefore, the entire family $\xi_{I_0}$ is algebraic over $\lambda_J$.

  The identity $\fR_0(\lambda_{I_0}) = 0$ is, by hypothesis, the unique relation satisfied by $\lambda_i$ over $k(\lambda_J)$.
  Therefore
  \begin{gather*}
    n_i
    = \dim_{k(\lambda_J)} \lambda_i
    = \dim_{k(\lambda_J,\xi_{I_0})} \lambda_i
    \leq \dim_{k(\xi_i)} \lambda_i
    \leq n_i.
  \end{gather*}
  In other words, $\lambda_i \in V^*(\xi_i)$ is generic over $k(\lambda_J,\xi_{I_0})$.

  Since $i \in I_0$ was arbitrary, $\lambda_{I_0} \in V^*(\xi_{I_0})$ is generic.
  On the other hand, it satisfies a unique algebraic relation over $k$, namely, $\fR_0$.
  Therefore the family $\xi_{I_0}$ is singly constrained, and $\fR_0 = \fR^k_{\xi_{I_0}}$.
  Since $X^*_i$ occurs in $\fR_0$ for every $i \in I_0$, the family is minimally constrained.
\end{proof}

\begin{thm}
  \label{thm:ResultantAbstract}
  Let $k$ be a field, and $\fR \in k[X^*_I]$.
  Let also $L \supseteq k$ be an algebraically closed extension of infinite transcendence degree.
  Then the following are equivalent:
  \begin{enumerate}
  \item The polynomial $\fR$ is a resultant in $X^*_I$ over $k$.
  \item The polynomial $\fR$ is a weak resultant in $X^*_I$ over $k$.
  \item Every irreducible factor of $\fR$ in $k[X^*_I]$ is of the form $\fR^k_{\xi_J}$, where $J \subseteq I$ and $\xi_J \in \bP^{n_J}(L)$ is minimally constrained.
  \end{enumerate}
  Moreover, if $n_i \neq 0$ for all $i$, we may extend each $\xi_J$ to a singly constrained family $\xi_I$, with the same resultant over $k$.
\end{thm}
\begin{proof}
  \begin{cycprf}
  \item Observed in \autoref{rmk:ResultantAbstract}.
  \item By \autoref{lem:ResultantAbstract}.
  \item[\impfirst]
    By \autoref{lem:ResultantWedgeVeronese}, \autoref{lem:ResultantWedgeCommute} and \autoref{rmk:ResultantReduction}.
  \end{cycprf}
\end{proof}

\begin{cor}
  \label{cor:ResultantAbstract}
  Let $A$ be a reduced ring.
  Then every weak resultant over $A$ is a resultant.
\end{cor}

\begin{cor}
  \label{cor:ResultantWedgeIrreducible}
  Let $i \in I$ and $J = I \setminus \{i\}$.
  Let $A$ be a unique factorisation domain and $\fR \in A[X^*_I] \setminus A[X^*_J]$ an irreducible resultant (over $A$).
  Let $d \geq 1$ and let $F_i = \sum_{|\alpha| = d} T^*_{i,\alpha} X_i^\alpha$ be an indeterminate polynomial of degree $d$.
  Then $F_i \wedge_i \fR \in A[T^*_i,X^*_J]$ is irreducible.
\end{cor}
\begin{proof}
  Let $k = \Frac(A)$.
  Then $\fR$ is irreducible over $k$.
  By \autoref{thm:ResultantAbstract}, there exists a field extension $L \supseteq k$, a subset $I' \subseteq I$ and a minimally constrained family $\xi_{I'} \in \bP(L)^I$ such that $\fR = \fR^k_{\xi_{I'}}$.
  Since $\fR \notin A[X^*_J]$, we have $i \in I'$, and we may assume that $I' = I$.
  Let $\zeta = \xi_i^{\otimes d}$.
  Then $F_i \wedge_i \fR = \fR^k_{\zeta,\xi_J}$, by \autoref{lem:ResultantWedgeVeronese}.
  In particular, $F_i \wedge_i \fR$ is irreducible in $k[T^*,X^*_J]$.
  Assume that $a \in A$ and $a \mid F_i \wedge_i \fR$ in $A[T^*,X^*_J]$.
  Since $F_i$ is indeterminate, we may substitute $(X_i^* \cdot X_i)^d$ for it.
  Then $a \mid (X_i^* \cdot X_i)^d \wedge_i \fR = \fR^d$ in $A[X^*_I]$.
  Therefore, $a$ is a unit, completing the proof.
\end{proof}

\section{Chow forms}
\label{sec:ChowForm}

In this section we fix an ambient projective dimension $n$, and let $X$ denote indeterminate homogeneous coordinates in $\bP^n$.
We let $X^*$ be the dual indeterminates, and for $i \in \bN$ we let $X^*_i$ be a copy of $X^*$.
For $\ell \in \bN$, we let $X^*_{\leq \ell}$ denote the family $(X^*_0,\ldots,X^*_\ell)$.
In the notation of \autoref{sec:Resultant}, this would be $X^*_I$, with $I = \{0,\ldots,\ell\}$.
Sometimes we would also consider the case where $\ell = -1$, in which case $X^*_{\leq \ell}$ is empty.

\begin{dfn}
  \label{dfn:ChowFormAlgebraicSet}
  Let $K$ be an algebraically closed field, and $\ell \geq -1$.
  Let $W \subseteq \bP^n(K)$ an algebraic set (Zariski closed, not necessarily irreducible) defined over $K$ and $\fC \in K[X^*_{\leq \ell}]$.
  We say that $\fC$ is a \emph{Chow form for $W$} in dimension $\ell$ if for every family $\lambda_{\leq \ell}$ of linear forms with coefficients in $K$:
  \begin{gather}
    \label{eq:ChowForm}
    \fC(\lambda_{\leq \ell}) = 0 \qquad \Longleftrightarrow \qquad W \cap V(\lambda_{\leq \ell}) \neq \emptyset.
  \end{gather}
  If $\fC$ is irreducible, then it is determined by \autoref{eq:ChowForm} up to a factor in $K^\times$.
  We then say that it is the \emph{Chow form of $W$}, and denote it by $\fC_W$.
\end{dfn}

There are several borderline cases that deserve special notice:
\begin{enumerate}
\item If $\dim W > \ell$, then $\fC$ is a Chow form in dimension $\ell$ for $W$ if and only if $\fC = 0$.
\item If $W = \emptyset$, then $\fC$ is a Chow form in dimension $\ell$ for $W$ if and only if $\fC \in K^\times$.
\item In dimension $\ell = -1$, this means that $\fC$ is a Chow form for $W \neq \emptyset$ if and only if $\fC = 0$, and for $\emptyset$ if and only if $\fC \in K^\times$.
\end{enumerate}

\begin{prp}
  \label{prp:ChowFormAlgebraicSet}
  Let $W \subseteq \bP^n(K)$ be an algebraic set, and $\fC \in K[X^*_{\leq \ell}]$ a Chow form for $W$ in dimension $\ell \geq 0$.
  \begin{enumerate}
  \item
    \label{item:ChowFormAlgebraicSetSubstitution}
    For any family of linear forms $\lambda_{<\ell}$,
    \begin{itemize}
    \item either $W \cap V(\lambda_{<\ell})$ is infinite and $\fC(\lambda_{<\ell},X^*) = 0$,
    \item or $W \cap V(\lambda_{<\ell})$ can be enumerated, possibly with repetitions, as $\bigl\{ [x_t] : t < d \bigr\}$, and $\fC(\lambda_{<\ell},X^*) = a \prod_{t<d} \, (X^* \cdot x_t)$ for some $a \in K^\times$.
      In particular, $d = \deg_{X^*_\ell} \fC$.
    \end{itemize}
    The second case holds whenever $\fC \neq 0$ and $\lambda_{<\ell}$ is free over the coefficients of $\fC$.
  \item
    \label{item:ChowFormAlgebraicSetWedge}
    As a polynomial in $X^*_\ell$, $\fC$ is splittable, and for any homogeneous polynomial $f \in K[X]_D$, the polynomial $f \wedge \fC \in K[X^*_{<\ell}]$ is a Chow form in dimension $\ell-1$ for $W \cap V(f)$.
  \end{enumerate}
\end{prp}
\begin{proof}
  For \autoref{item:ChowFormAlgebraicSetSubstitution}, observe first that by definition, $\fC(\lambda_{<\ell},X^*) \in K[X^*]$ is a Chow form in dimension zero for $W_0 = W \cap V(\lambda_{<\ell})$.
  If $W_0$ is infinite, then $\dim W_0 > 0$ and $\fC(\lambda_{<\ell},X^*) = 0$.
  If $W_0$ is finite, then $\fC(\lambda_{\leq \ell}) = 0$ if and only if $\lambda_\ell$ vanishes at some point of $W_0$.
  It follows that the irreducible factors of $\fC(\lambda_{<\ell},X^*)$ are exactly $X^* \cdot x$ for $[x] \in W_0$.

  For \autoref{item:ChowFormAlgebraicSetWedge}, consider two cases.
  If $\fC = 0$, then it is splittable, $\dim W > \ell$, and $\dim W \cap V(f) > \ell-1$.
  Therefore $f \wedge \fC = 0$ is indeed a Chow form for $W \cap V(f)$.
  If $\fC \neq 0$, then it is splittable in $X^*_\ell$ by \autoref{item:ChowFormAlgebraicSetSubstitution}.
  Let $\fD = f \wedge \fC \in K[X^*_{<\ell}]$.
  Let $\lambda_{<\ell}$ be any linear forms, and let $W_0$ and $\fC(\lambda_{<\ell},X^*) = a \prod_{t<d} \, (X^* \cdot x_t)$ be as in \autoref{item:ChowFormAlgebraicSetSubstitution}.
  Then
  \begin{gather*}
    \fD(\lambda_{<\ell}) = f \wedge \fC(\lambda_{<\ell},X^*) = a^{\deg f} \prod_{t<d} f(x_t).
  \end{gather*}
  In particular, $W \cap V(f) \cap V(\lambda_{<\ell}) \neq \emptyset$ if and only if $f$ vanishes at some point of $W_0$, if and only if $\fD(\lambda_{<\ell}) = 0$.
\end{proof}

\begin{lem}
  \label{lem:ChowFormAlternating}
  Let $\xi_I \in \bP^n(K)^I$ be singly constrained over $k$ and $\fC = \fR^k_{\xi_I}$.
  Then the following are equivalent:
  \begin{enumerate}
  \item There exists a point $\xi \in \bP^n(K)$ such that $\xi_i = \xi$ for all $i$ and $\fC$ is a Chow form for $W = \loc_k \xi$.
  \item The polynomial $\fC$ is a Chow form for some algebraic set $W \subseteq \bP^n(K)$.
  \item The polynomial $\fC$ is alternating.
  \end{enumerate}

  In particular, every irreducible variety admits a Chow form.
\end{lem}
\begin{proof}
  \begin{cycprf}
  \item
    Immediate.
  \item
    Let $\mu_I$ be linear forms.
    Let $i \in I$ and $j \in J = I \setminus \{i\}$.
    Then $V(\mu_i,\mu_J) = V(\mu_i+a\mu_j,\mu_J)$ for every $a \in K$.
    By \autoref{dfn:ChowFormAlgebraicSet}, $\fC(\lambda_i + a \lambda _j,\lambda_J) = 0$ if and only if $\fC(\lambda_i,\lambda_J) = 0$.
    Therefore $\fC(X^*_i + Y X^*_j, X^*_J)$ cannot depend on $Y$, and must be equal to $\fC(X^*_i, X^*_J)$
  \item[\impfirst]
    Let $\lambda_I$ be a generic in $V^*(\xi_I)$, so $\fC(\lambda_I) = 0$ is the unique algebraic relation over $k$ satisfied by $\lambda_I$.
    Let $i \in I$ and $j \in J = I \setminus \{i\}$.
    Since $\fC$ is alternating and not a constant, $X^*_i$ occurs in $\fC$, so $\lambda_J$ is free over $k$.
    By \autoref{thm:ResultantSplitting}, $\fC(X^*_i,\lambda_J)$ splits as $\prod_{t < D} \, (X^*_i \cdot x_{i,t})$, where $\xi_i = [x_{i,0}]$.
    By alternation,
    \begin{gather*}
      \prod_{t < D} \, (\lambda_i + \lambda_j)(x_{i,t}) = \fC(\lambda_i + \lambda_j,\lambda_J) = \fC(\lambda_i,\lambda_J) = 0.
    \end{gather*}
    Since $\lambda_i$ is generic in $V^*(\xi_i)$, and $\lambda_j$ is generic in $V^*(\xi_j)$, this is only possible if $\xi_i = \xi_j$.
    Therefore, the sequence $\xi_I$ is constant, equal to some $\xi \in \bP^n(K)$.
    Let $W = \loc_k \xi$.

    Let us consider a family $\mu_I \in \bigl( K[X]_1 \bigr)^I$.
    If $\fC(\mu_I) = 0$, then $\mu_I$ specialises $\lambda_I$ over $k$.
    This can be extended to a specialisation $\mu_I,\zeta$ of $\lambda_I,\xi$, so $\zeta \in W \cap V(\mu_I)$.
    Conversely, if $\zeta \in W \cap V(\mu_I)$, then $\zeta$ specialises $\xi$.
    Since $\lambda_I \in V^*(\xi)^I$ is generic, the family $\mu_I,\zeta$ is a specialisation of $\lambda_I,\xi$, and in particular, $\fC(\mu_I) = 0$.
    Thus $\fC$ is a Chow form for $W$.
  \end{cycprf}
\end{proof}

\begin{dfn}
  \label{dfn:ChowFormAbstract}
  Let $A$ be a ring, and $\ell \geq -1$.
  A \emph{Chow form} in dimension $\ell$ is an alternating resultant $\fC \in A[X^*_{\leq \ell}]$.

  If $\ell \geq 0$, then $\fC$ is a Chow form of \emph{degree} $d$ if it is of degree $d$ as a polynomial in $X^*_0$.
  If $\fC = 0$, or if $\ell = -1$, then we do not define the degree.
\end{dfn}

\begin{thm}
  \label{thm:ChowFormAbstract}
  Let $K$ be an algebraically closed field, and $\fC \in K[X^*_{\leq \ell}] \setminus \{0\}$ for $\ell \geq 0$.
  Then $\fC$ is Chow form if and only if it is a Chow form for some algebraic set $W \subseteq \bP^n(K)$.
  Moreover, the irreducible factors of $\fC$ are exactly the Chow forms $\fC_U$, where $U$ varies over the irreducible components of $W$.
  In particular, the set $W$ is determined by $\fC$, is defined over $K$, and is of pure dimension $\ell$.
\end{thm}
\begin{proof}
  Assume first that $\fC$ is a Chow form, and let $\fC = \prod_t \fC_t$ be its factorisation in $K[X^*_{\leq \ell}]$.
  By \autoref{thm:ResultantAbstract}, each $\fC_t$ is of the form $\fR^k_{\xi_J}$ where $J \subseteq I$ and $\xi_J$ is minimally constrained.
  By \autoref{lem:AlternatingFactor}, $\fC_t$ is alternating.
  Therefore $J = I$, and by \autoref{lem:ChowFormAlternating}, $\fC_t = \fC_{U_t}$ for some $U_t \subseteq \bP^n(K)$, irreducible of dimension $\ell$.
  It follows that $\fC$ is a Chow form for $W = \bigcup U_t$.

  Conversely, assume that that $\fC$ is a Chow form for $W$.
  By \autoref{prp:ChowFormAlgebraicSet}, $\fC$ is a weak resultant, and by \autoref{thm:ResultantAbstract}, it is a resultant.
  By \autoref{lem:ChowFormAlternating}, it is alternating.
\end{proof}

\begin{exm}
  \label{exm:ChowForm}
  Let us consider a few obvious cases.
  Assume that $\fC \in K[X^*_{\leq \ell}] \setminus \{0\}$.
  \begin{itemize}
  \item
    If $\ell > n$, then $\fC$ is a Chow form if and only if it belongs to $K^\times$.
    Indeed, the only alternating polynomials in $X^*_{\leq \ell}$, and \textit{a fortiori} the only Chow forms in dimension $\ell$, are the constants.
  \item
    If $\ell = n$, then $\fC$ is a Chow form if and only if $\fC = a \det(X^*_{\leq n})^D$ for some $a \in K^\times$ and $D \in \bN$.
    Indeed, the determinant is the Chow form of $\bP^n$, and it is the only irreducible alternating polynomial in $X^*_{\leq n}$.
  \item
    If $\ell = n-1$, then $\fC$ is a Chow form if and only if it is of the form $g(X^*_0 \wedge \cdots \wedge X^*_{n-1})$, where $g \in K[X] \setminus \{\emptyset\}$ is homogeneous.
    Indeed, every alternating polynomial is of this form, and conversely, $g(X^*_0 \wedge \cdots \wedge X^*_{n-1}) = \pm g \wedge \det$ is the Chow form of $V(g)$.
  \item
    If $\ell = 0$, then $\fC \in k[X^*]$ (identifying $X^*$ with $X^*_0$) is a Chow form if and only if it is splittable.
    Thus, $\fC$ is a Chow form for an algebraic set $W$ if and only if $\fC$ splits as $a \prod_{j < d} x_j$, in the sense of \autoref{conv:PointsAsPolynomials}, and $W = \bigl\{ [x_j] : j < d \bigr\}$.

    Indeed, the zeros of $\fC$ are exactly those $\lambda$ that vanish on some point of $W$.
  \item
    If $\ell = -1$, then $\fC \in K^\times$, and conversely, every constant in $K^\times$ is indeed a Chow form for the empty set.
  \end{itemize}
  The first three cases are covered by the discussion following \autoref{lem:AlternatingDeterminant}.
\end{exm}

\begin{exm}
  \label{exm:NonChowForm}
  We have seen that for $\ell \geq n-1$, every homogeneous alternating polynomial is a Chow form, and for $\ell \leq 0$, every splittable polynomial is one.
  However, in general, an alternating splittable polynomial need not be a Chow form.

  Indeed, consider the minimal setting which does not fit in any of the easy cases, namely $n = 3$ and $\ell = 1$.
  In this case, $X^*$ consists of four indeterminates.
  Let $g(X^*_0,X^*_1)$ be any non-degenerate alternating bilinear form on four-dimensional space.
  Then it is an alternating, homogeneous (of degree one) splittable (since it is already linear) polynomial.
  However, it is not a Chow form (or else it would be the Chow form of a degree one curve, i.e., of a line, but those are degenerate bilinear forms).
\end{exm}

\begin{rmk}
  Let $\fC$ be a Chow form for $W$, and let $\fC^\sharp(\bY^*,Y)$ be as per \autoref{dfn:AlternatingSharp}.
  Let $\xi = [x] \in \bP^n(K)$.
  It is then easy to check that $\xi \in W$ if and only if $\fC^\sharp(\bY^*,\bY^* \cdot x) = 0$.

  Consider the polynomial $\fC^\sharp(\bY^*,\bY^* \cdot X)$.
  Viewed as a polynomial in the indeterminates $\bY^*$, its coefficients are polynomials in $X$, homogeneous of degree $d$, and together they define $W$.
\end{rmk}

\begin{ntn}
  \label{ntn:ChowFormWedge}
  Let $\fC \in A[X^*_{\leq \ell}]$ be a Chow form.
  If $\ell \geq 0$, then for a homogeneous polynomial $f \in A[X]$ we shall denote $f \wedge_\ell \fC \in A[X^*_{\leq\ell-1}]$ by $f \wedge \fC$ (see \autoref{ntn:MultiWedge}, and compare also with \autoref{prp:ChowFormAlgebraicSet}\autoref{item:ChowFormAlgebraicSetWedge}).

  Since $f \wedge \fC$ is again a Chow for, we may iterate the operation.
  Thus, if $k \leq \ell + 1$ and $F = (f_i : i < k)$ is a sequence of homogeneous polynomials (of varying degrees), we let
  \begin{gather*}
    F \wedge \fC = f_0 \wedge \bigl(f_1 \wedge \ldots (f_{k-1} \wedge \fC) \ldots \bigr).
  \end{gather*}
  When $\fC_{\bP^n}$ is the normalised Chow form of $\bP^n$, namely the determinant form, we allow ourselves to omit it, writing
  \begin{gather*}
    F^\wedge = f_0 \wedge \cdots \wedge f_{k-1} = F \wedge \fC_{\bP^n}.
  \end{gather*}
  This coincides with the Macaulay resultant of the family $F$, see \cite{Macaulay:Elimination}.
\end{ntn}

If $\deg f_i = 0$, then $F \wedge \fC = f_i^D$, where $D = \deg \fC \prod_{j \neq i} \deg f_j$.
In particular, if $\deg f_i = \deg f_j = 0$ for $i \neq j$ then $F \wedge \fC = 1$.

Let $\fC \in A[X^*_{\leq \ell}]$ be a Chow form of degree $D$ in dimension $\ell \geq 0$.
\begin{enumerate}
\item
  \label{item:ChowFormWedgeIrreducible}
  Assume that $\fC$ is irreducible over a unique factorisation domain $A$.
  Let $f = \sum T^*_\alpha X^\alpha$ be indeterminate polynomial of degree $d \geq 1$.
  By \autoref{cor:ResultantWedgeIrreducible}, $f \wedge \fC$ is irreducible in $A[T^*,X^*_{<\ell}]$, i.e., as a Chow form in dimension $\ell-1$ over $A[T^*]$.
  The same follows for an iterated wedge operation with several indeterminate polynomials of non-zero degrees.
\item
  \label{item:ChowFormWedgePermutation}
  Let $f \in A[X]_d$, $g \in A[X]_e$, and assume that $\ell \geq 1$.
  By alternation we have $\fC(\ldots, X^*_{\ell-1}, X^*_\ell) = \fC(\ldots, X^*_\ell, -X^*_{\ell-1}) = (-1)^D\fC(\ldots, X^*_\ell, X^*_{\ell-1})$.
  Therefore
  \begin{gather*}
    f \wedge g \wedge \fC = (-1)^{deD} g \wedge f \wedge \fC.
  \end{gather*}
  It follows that if $F = (f_i : i < k)$ is a family of $k \leq \ell+1$ homogeneous polynomials, $d = \prod \deg f_i$, $\sigma \in \fS_k$, and $F^\sigma = (f_{\sigma(i)} : i < k)$, then
  \begin{gather*}
    F^\sigma \wedge \fC = \sgn \sigma^{dD} F \wedge \fC.
  \end{gather*}
\item
  \label{item:ChowFormWedgeIteratedIntersection}
  Assume that $A = K$ is a field, and $\fC$ is a Chow form for some algebraic set $W$.
  Let $k \leq \ell + 1$ and $F = (f_i : i < k)$ be homogeneous polynomials over $K$.
  Then, applying induction to \autoref{prp:ChowFormAlgebraicSet}\autoref{item:ChowFormAlgebraicSetWedge}, the iterated wedge operation $F \wedge \fC \in K[X^*_0,\ldots,X^*_{\ell-k}]$ is a Chow form in dimension $\ell-k$ for $W \cap V(F)$.
  When $k = \ell+1$, this means that $F \wedge \fC = 0$ if and only if $W \cap V(F) \neq \emptyset$.
\end{enumerate}

\section{Generic projections}

Let $K$ be algebraically closed, and let $\fC \in K[X^*_{< \nu}]$ be a Chow form over $K$, in dimension $\ell = \nu-1$, of degree $D$.
We may assume that it is a Chow form for $W \subseteq \bP^n(K)$, in the sense of \autoref{dfn:ChowFormAlgebraicSet}.
Let $Y = (Y_i : i \leq \nu)$ be new indeterminates, let $\lambda = (\lambda_i : i \leq \nu)$ be linear forms, and let $\fC^\sharp$ be as per \autoref{dfn:AlternatingSharp}.
Define
\begin{gather*}
  P_\lambda(Y)
  = \fC^\sharp(\lambda,Y)
  = (-1)^{(\nu-j)D} \fC\left( \lambda_i - \frac{Y_i}{Y_j} \lambda_j : i \leq \nu, \ i \neq j \right) Y_j^D,
\end{gather*}
where the last equality holds for any $j \leq \nu$ by \autoref{lem:AlternatingSharp}\autoref{item:AlternatingSharpAlternate}.
This is a homogeneous polynomial of degree $D$ in $Y$, whose coefficients are polynomials in the coefficients of $\lambda$.
The coefficient of $Y_j^D$ in $P_\lambda$ is $(-1)^{(\nu-j)D} \fC(\lambda_{\neq j})$.

\begin{lem}
  With these definitions, $P_\lambda = 0$ if and only if $V(\lambda) \cap W \neq \emptyset$.
  If $V(\lambda) \cap W \neq \emptyset$, then $[x] \mapsto [\lambda x]$ defines a map that we may also denote by $\lambda\colon W \rightarrow \bP^\nu$, and $\lambda W = V(P_\lambda)$.
\end{lem}
\begin{proof}
  Since $K$ is algebraically closed, we may assume that everything happens in $K$.
  If $V(\lambda) \cap W \neq \emptyset$, then $V( Y_\nu \lambda_i - Y_i \lambda_\nu : i < \nu) \cap W \neq \emptyset$ regardless of $Y$, so $P_\lambda = 0$.
  Assume therefore that $V(\lambda) \cap W = \emptyset$.
  Then $\lambda x \neq 0$ for every $[x] \in W$, and $[\lambda x] \in \bP^\nu$.
  Let $[y] \in \bP^\nu$, so $y_j \neq 0$ for some $j$.
  Then $P_\lambda(y) = 0$ if and only if there exists $[x] \in W$ such that
  \begin{gather*}
    y_j \lambda_i x = y_i \lambda_j x
  \end{gather*}
  for all $i \leq \nu$.
  If $x$ is such, then $\lambda_j x \neq 0$ (or else $\lambda x = 0$, which cannot be), and $\lambda[x] = [y]$.
  Conversely, if $[x] \in W$, then $P_\lambda(\lambda x) = 0$.
  Therefore,
  \begin{gather*}
    V(P_\lambda) = \lambda W.
  \end{gather*}
  In particular, $V(P_\lambda) \neq \bP^\nu$, so $P_\lambda \neq 0$.
\end{proof}

\begin{prp}
  \label{prp:LinearProjectionIntersection}
  Let $f = (f_i : i < \nu)$ be a family of homogeneous polynomials in $Y$, and $\lambda$ linear forms as above.
  Then
  \begin{gather*}
    f \wedge P_\lambda = (f \circ \lambda) \wedge \fC.
  \end{gather*}
  Moreover, if $\fC$ is irreducible in $K[X^*_{<\nu}]$, and $f$ and $\lambda$ are indeterminate (in their respective degrees), then $P_\lambda$ is irreducible in $K[\lambda,Y]$ and $f \wedge P_\lambda = (f \circ \lambda) \wedge \fC$ is irreducible in $K[\lambda,f]$.
\end{prp}
\begin{proof}
  In order to prove either part, We may assume that $\fC$ is irreducible and that $\lambda$ and $f$ are indeterminate.
  Then $\fC = \fC_W$ for an irreducible $W$.
  By \autoref{lem:AlternatingSharp}\autoref{item:AlternatingSharpIrreducible}, $P_\lambda$ is irreducible in $K[\lambda,Y]$, and therefore in $L[Y]$, where $L$ is the algebraic closure of $K(\lambda)$.
  By \autoref{cor:ResultantWedgeIrreducible}, $f \wedge P_\lambda$ is irreducible in $L[f]$, and therefore in $K[\lambda,f]$.

  Consider a specialisation of $f$ to a family $g = (g_{<\nu})$, where $g_i \in L[X]_{d_i}$.
  Then $g \wedge P_\lambda = 0$ if and only if there exists $[y] \in V(g,P_\lambda)$.
  Since $V(P_\lambda) = \lambda W$, this is equivalent to the existence of $[x] \in W$ such that $[\lambda x] = [y] \in V(g)$, or yet equivalently, to the existence of $[x] \in W \cap V(g \circ \lambda)$.
  In other words, $g \wedge P_\lambda = 0$ if and only if $(g \circ \lambda) \wedge \fC = 0$.
  Stated equivalently, the polynomials $f \wedge P_\lambda$ and $(f \circ \lambda) \wedge \fC$ (both in $L[f]$) have the same zeros.
  Since the former is irreducible, we must have
  \begin{gather*}
    a (F \wedge P_\lambda)^r = (F \circ \lambda) \wedge \fC
  \end{gather*}
  for some $a \in L^\times$ and $r \geq 1$.
  A comparison of degrees yields that $r = 1$.
  Since $f$ were assumed indeterminate, we may substitute powers of linear polynomials, and reduce the calculation of $a$ to the linear case.
  In fact, we may substitute $f_i = Y_i$ for $i < \nu$.
  Then $Y_0 \wedge \ldots \wedge Y_{\nu-1} = (-1)^\nu (0,\ldots,0,1)$ and
  \begin{gather*}
    f \wedge P
    = (-1)^{\nu D}  P \wedge Y_0 \wedge \ldots \wedge Y_{\nu-1}
    = P( 0,\ldots,0, 1 )
    = \fC(\lambda_{<\nu})
    = (f \circ \lambda) \wedge \fC.
  \end{gather*}
  Therefore $a = 1$, which completes the proof.
\end{proof}

\bibliographystyle{begnac}
\bibliography{begnac}

\end{document}